\documentclass[11pt]{amsart}

\usepackage{amssymb,amsmath,amsthm,amsfonts}
\usepackage{graphicx}
\usepackage{bbm}
\usepackage{hyperref}
\usepackage{tikz}
\usepackage{tikz-cd}
\usepackage[utf8]{inputenc}
\usepackage{mathtools}
\usepackage{yhmath}
\usepackage{theoremref}
\usepackage{cite}
\usepackage{comment}
\usepackage[margin=1in]{geometry}
\usepackage{enumerate}

\newtheorem{introthm}{Theorem}

\newtheorem{introcol}[introthm]{Corollary}
\newtheorem{introprop}[introthm]{Proposition}

\newtheorem{thm}{Theorem}[section]
\newtheorem{lemma}[thm]{Lemma}
\newtheorem{col}[thm]{Corollary}
\newtheorem{prop}[thm]{Proposition}
\newtheorem{fact}[thm]{Fact}

\theoremstyle{definition}
\newtheorem{defn}[thm]{Definition}

\theoremstyle{remark}
\newtheorem{ex}[thm]{Example}
\newtheorem{rmk}[thm]{Remark}

\DeclareMathOperator{\re}{Re}
\DeclareMathOperator{\im}{Im}
\DeclareMathOperator{\id}{id}
\DeclareMathOperator{\tr}{tr}
\DeclareMathOperator{\sa}{sa}
\DeclareMathOperator{\Spec}{Spec}
\DeclareMathOperator{\ev}{ev}
\DeclareMathOperator{\Th}{Th}
\DeclareMathOperator{\Leb}{Leb}

\newcommand{\cU}{\mathcal{U}}
\newcommand{\cV}{\mathcal{V}}

\newcommand{\N}{\mathbb{N}}
\newcommand{\C}{\mathbb{C}}
\newcommand{\R}{\mathbb{R}}
\newcommand{\Q}{\mathbb{Q}}
\newcommand{\M}{\mathbb{M}}

\DeclarePairedDelimiter{\norm}{\lVert}{\rVert}

\begin{document}
	
	\title[Elementary equivalence and disintegration]{Elementary equivalence and disintegration of tracial von Neumann algebras}
	
	\author{David Gao}
	\address{Department of Mathematical Sciences, UCSD, 9500 Gilman Dr, La Jolla, CA 92092, USA}
	\email{weg002@ucsd.edu}
	\urladdr{https://sites.google.com/ucsd.edu/david-gao/home}
	
	\author{David Jekel}
	\address{Department of Mathematical Sciences, University of Copenhagen, Universitetsparken 5, 2100 Copenhagen {\O}, Denmark}
	\email{daj@math.ku.dk}
	\urladdr{https://davidjekel.com}
	
	\begin{abstract}
		We prove an analog of the disintegration theorem for tracial von Neumann algebras in the setting of elementary equivalence rather than isomorphism, showing that elementary equivalence of two direct integrals of tracial factors implies fiberwise elementary equivalence under mild, and necessary, hypotheses.  This verifies a conjecture of Farah and Ghasemi.  Our argument uses a continuous analog of ultraproducts where an ultrafilter on a discrete index set is replaced by a character on a commutative von Neumann algebra, which is closely related to Keisler randomizations of metric structures.  We extend several essential results on ultraproducts, such as {\L}o{\'s}'s theorem and countable saturation, to this more general setting.
	\end{abstract}
	
	\maketitle
	
	\section{Introduction}
	
	\subsection{Motivation}
	
	Our results are motivated by the problem of classifying tracial von Neumann algebras up to elementary equivalence.  Elementary equivalence is a central notion in model theory, which means that two objects have the same first-order theory.  An equivalent characterization that is more familiar to operator algebraists follows from the Keisler-Shelah theorem:  Two tracial von Neumann algebras $M$ and $N$ are elementarily equivalent (written $M \equiv N$) if and only if there is some ultrafilter $\cU$ on some index set such that the ultrapowers $M^{\cU}$ and $N^{\cU}$ are isomorphic.
	
	Von Neumann \cite{vonNeumann1949reduction} showed that a von Neumann algebra with separable predual can be decomposed as a direct integral $M = \int_\Omega^{\oplus} M_\omega\,d\omega$ over some measure space $\Omega$ such that the center $Z(M) \subseteq M$ agrees with $L^\infty(\Omega) \subseteq \int_\Omega^{\oplus} M_\omega\,d\omega$, and hence reduced the classification of separable von Neumann algebras to that of factors.  In the case where the von Neumann algebra admits a faithful normal tracial state, the factors are either matrix algebras or $\mathrm{II}_1$ factors.  The classification of factors, even in the $\mathrm{II}_1$ case, is an extremely challenging task.  While Murray and von Neumann distinguished group von Neumann algebras $L(G)$ for free groups versus amenable groups, it was not until the work of McDuff that infinitely many non-isomorphic $\mathrm{II}_1$ factors were shown to exist \cite{McD69a,McD69b}.  Recent breakthroughs have classified von Neumann algebras associated to large classes of groups using deformation/rigidity techniques (see e.g.~\cite{IPV2013,CIOS2023}), but are still very far from general $\mathrm{II}_1$ factors.
	
	Through the study of ultrapowers and later the introduction (officially) of continuous model theory for von Neumann algebras in \cite{FHS2014a}, there was growing interest in classification of von Neumann algebras up to elementary equivalence, and many of the developments paralleled the early history of classification up to isomorphism.  Farah, Hart, and Sherman showed that property Gamma is axiomatizable, and hence amenable von Neumann algebras and free group von Neumann algebras are not elementarily equivalent \cite[\S 3.2.2]{FHS2014b}; preservation of Gamma under elementary equivalence also follows from the earlier work of Fang, Ge, and Li \cite[Corollary 5.2]{FGL2006}.  Boutonnet, Chifan, and Ioana  showed that McDuff's family of $\mathrm{II}_1$ factors are not elementarily equivalent \cite{BCI17}; see also \cite{GH2017}.  In the non-Gamma setting, the first examples proved to be non-elementarily equivalent were given in \cite{CIKE23}, and further examples were obtained in \cite{KEP2023}.
	
	However, the more fundamental question about the analog of von Neumann's disintegration theorem for elementary equivalence was not fully addressed (and similarly, there is not much written about classification of tracial von Neumann algebras with nontrivial center in practice, except for \cite{CFQH2024}).  Farah and Ghasemi showed, in general, that the theory of a direct integral is uniquely determined by the theories of the integrands \cite[Corollary 3.8]{FG24} (see also \cite[Theorem 18.3]{BYIT2024}).  In the case where the integral is over an atomless probability space, this also follows from the work of Ben Yaacov on randomizations of metric structures \cite{BY13}.
	
	Our goal in this paper is to establish a converse recovering the theory of the integrands from the theory of the integral for tracial von Neumann algebras, which was conjectured by Farah and Ghasemi in an earlier version of \cite{FG24}.  In other words, we will show that the integrands in the direct integral can be recovered up to elementary equivalence if the integral is known up to elementary equivalence.  After the first draft of this paper appeared, Ben Yaacov, Ibarluc{\'i}a, and Tsankov added a general version of this result to their systematic treatment of direct integrals, affine logic, and continuous logic; see \cite[Theorem 20.13]{BYIT2024}.
	
	\subsection{Results}
	
	We require a slight technical hypothesis to rule out different multiplicities of the same factor occurring in the diffuse part of the measure space in the direct integral decomposition.  For example, suppose that $(\Omega,\mu)$ is diffuse (i.e.\ has no atoms) and $(M_\omega)_{\omega \in \Omega}$ are a measurable field of pairwise non-elementarily equivalent tracial factors.  Let $(\tilde{\Omega},\tilde{\mu}) = (\Omega \sqcup \Omega,\frac{1}{2} \mu \oplus \frac{1}{2} \mu)$, and consider the measurable field $(N_\omega)_{\omega \in \tilde{\Omega}}$ on $\tilde{\Omega}$ consisting of two copies of $(M_\omega)_{\omega \in \Omega}$.  Then it follows from \cite[Corollary 3.8]{FG24} that $\int_\Omega (M_\omega,\tau_\omega) \,d\mu(\omega)$ is elementarily equivalent to $\int_{\tilde{\Omega}} N_\omega\,d\tilde{\mu}(\omega) \cong \int_\Omega (M_\omega \oplus M_\omega,\frac{1}{2} \tau_\omega \oplus \frac{1}{2} \tau_\omega)\,d\mu(\omega)$, but there is no measurable isomorphism $(\Omega,\mu) \to (\Omega \sqcup \Omega,\frac{1}{2} \mu \oplus \frac{1}{2} \mu)$ exhibiting elementary equivalence fiberwise because in $\Omega$ all the elementary equivalence classes are distinct but in $\tilde{\Omega}$ there are two copies of each elementary equivalence class.  Hence, there is an unavoidable ambiguity due to the different possibly multiplicities of each elementary equivalence class over the underlying measure space.  To rule out such examples, we add an assumption of ``infinite multiplicity'' of each fiber in the direct integral decomposition.
	
	Given a probability space $(\Omega,\mu)$ and a mapping $F$ from $\Omega$ into some space $\Omega'$, we say that that $(\Omega,\mu,F)$ is \emph{stable under tensorization with $L^\infty[0,1]$}, if there is a measurable isomorphism $\Phi: (\Omega,\mu) \to (\Omega \otimes [0,1], \mu \otimes \operatorname{Leb})$ such that $F \circ \pi_1 \circ \Phi = F$, where $\pi_1$ is the first coordinate projection (see Lemma \ref{lem: diffuse fibers} for equivalent characterizations of this property).  Note that the example in preceding paragraph created $\tilde{\Omega}$ by tensorizing $\Omega$ with a two-point space.  Of course, the measure space $\Omega$ in a direct integral decomposition can be expressed as a direct sum of an atomic one and an atomless (diffuse) one, and we only want to assume stability under tensorization for the diffuse part.  Hence, our main result can be stated as follows; note that the converse to Theorem \ref{introthm: EE disintegration} follows from \cite[Corollary 3.8]{FG24}.  See also Remark \ref{rmk: Thm A general case} for what happens without the stability assumption.
	
	\begin{introthm} \label{introthm: EE disintegration}
		Let $(\Omega_1,\mu_1)$ and $(\Omega_2,\mu_2)$ be standard Borel probability spaces, and let $(M,\tau) = \int_{\Omega_1}^{\oplus} (M_\omega,\tau_\omega) \,d\mu_1(\omega)$ and $(N,\sigma) = \int_{\Omega_2}^{\oplus} (N_\omega,\sigma_\omega) \,d\mu_2(\omega)$ be direct integrals of separable tracial factors.  Suppose that the diffuse part of each decomposition is stable under tensorization with $L^\infty[0,1]$.  If $(M,\tau) \equiv (N,\sigma)$, then there exists an isomorphism of probability spaces (i.e.\ a measure-preserving map with measurable inverse up to null sets) $f: \Omega_1 \to \Omega_2$ such that $(M_\omega, \tau_\omega) \equiv (N_{f(\omega)},\tau_{f(\omega)})$ for almost every $\omega \in \Omega_1$.
	\end{introthm}
	
	To give a simple example, we can obtain continuum many non-elementarily equivalent tracial von Neumann algebras with diffuse center as follows.  Let $P$ and $Q$ be two non-elementarily equivalent $\mathrm{II}_1$ factors.  Let $M_c$ be the direct integral over $[0,1]$ where the fibers on $[0,c]$ are $P$ and the fibers on $(c,1]$ are $Q$.  Theorem \ref{introthm: EE disintegration} shows that $M_c$ and $M_{c'}$ are not elementarily equivalent for $c \neq c'$.
	
	Our approach is based on constructing the ``ultrafibers'' out of a direct integral $(M,\tau) = \int_{\Omega}^{\oplus} (M_\omega,\tau_\omega) \,d\mu(\omega)$ associated to characters on $L^\infty(\Omega)$.  While it is not possible to obtain the fiber $(M_\omega,\tau_\omega)$ canonically as a $\mathrm{C}^*$-quotient of $M$ since a point evaluation is not necessarily well-defined on $L^\infty(\Omega)$, we can obtain a $\mathrm{C}^*$-quotient corresponding to each character on $L^\infty(\Omega)$.  If $\Omega$ was a discrete measurable space $I$, so that $L^\infty(\Omega) = \ell^\infty(I)$ for the index set $I$, then a character would always be given by an ultrafilter on $I$, and then our ultrafiber would reduce to the ultraproduct of $(M_i,\tau_i)_{i \in I}$ with respect to this ultrafilter.  As we will see, there is no need to assume that $M$ is a direct integral in the classical sense; we can replace $L^\infty(\Omega)$ by any central von Neumann subalgebra $N$ of a tracial von Neumann algebra $M$, and so obtain an ultrafiber of $M$ over $N$ for any character on $N$ (see \S \ref{subsec: ultrafiber}). 
	The possibility of this construction was already implicit in Wright's 1954 work on the quotients of $\mathrm{AW}^*$-algebras by maximal ideals \cite{Wri54}.  See also \cite[\S A.3]{SS08}.  The ultrafibers studied here also closely relate with Keisler and Ben Yaacov's work on randomizations of structures \cite{Keisler1999,BY13}.  Indeed, on careful inspection, the ultrafibers are exactly the fibers in the canonical realization of \cite[Theorem 3.11]{BY13}.  Hence, our work gives a more operator-algebraic approach to randomizations for tracial von Neumann algebras.
	
	Note in addition that \cite{BYIT2024} developed direct integrals in the non-separable setting more generally by relating affine logic with continuous logic, and in particular showed the existence of direct integral decompositions for arbitrary (including non-separable) tracial von Neumann algebras \cite[Corollary 29.10]{BYIT2024}.  See \cite[\S 9]{BYIT2024} for a discussion of combining direct integrals and ultraproducts.  See also \cite[\S 20]{BYIT2024} for a general version of the result on unique distribution of theories, which works for affine theories that are \emph{simplicial}.  As our goal is to present these results for an operator algebraic audience, we focus here on the separable case.
	
	Although here we argue using ultrafibers in order to minimize the model theory background needed to understand the proof, we also point out another, more model-theoretic way to prove Theorem \ref{introthm: EE disintegration} in the diffuse case using Ben Yaacov's disintegration theorem \cite[Theorem 3.32]{BY13}.  Ben Yaacov's theorem deals with a two-sorted structure with an auxiliary sort representing the probability space used in the randomization (for instance, the space $(\Omega,\mu)$ in a direct integral decomposition). 
	Thus, to apply this result, one seeks to upgrade the plain structure of a tracial von Neumann algebra into a structure with an auxiliary sort for the center, which is a commutative tracial von Neumann algebra and hence corresponds to a probability space.  The key point for this approach (which distinguishes tracial von Neumann algebras from certain other metric structures) is that the center is a definable set and the center-valued trace is a definable function, which follows from Dixmier averaging as in \cite[Lemma 4.2]{FHS13}. Hence, the new upgraded structure is definable in terms of the plain tracial von Neumann algebra structure.  From this one can deduce that the center-valued interpretations of formulas are definable functions by similar reasoning as \cite[Lemma 3.13]{BY13}.  We also remark that the disintegration result in \cite[Theorem 20.13]{BYIT2024} holds in greater generality, including situations where the fiberwise interpretations of formulas cannot be expressed as definable functions with respect to the original language in any sense; such definability is thus a stronger conclusion that requires stronger hypotheses.
	
	The special case of Theorem \ref{introthm: EE disintegration} when all the fibers are the same, so that the direct integral over $\Omega$ reduces to $L^\infty(\Omega,\mu) \overline{\otimes} M$, is of particular note.
	
	\begin{introcol} \label{introcol: Ozawa's question EE}
		Let $(M,\tau_M)$ and $(N,\tau_N)$ be $\mathrm{II}_1$ factors, and let $(\Omega,\mu)$ be a standard Borel probability space.  If $L^\infty(\Omega,\mu) 
		\overline{\otimes} (M,\tau_M) \equiv L^\infty(\Omega,\mu) \overline{\otimes} (N,\tau_N)$, then $(M,\tau_M) \equiv (N,\tau_N)$.
	\end{introcol}
	
	This is related to the problem, which has become known as Ozawa's question due a discussion on MathOverflow, of whether $L^\infty(\Omega,\mu) \overline{\otimes} M  \cong L^\infty(\Omega,\mu) \overline{\otimes} N$ implies $M \cong N$ in general (this is an exercise if we assume that $M$ and $N$ are separable); our result solves the analogous question for elementary equivalence.  In fact, we are able to give a simpler and more ``algebraic'' proof of Corollary \ref{introcol: Ozawa's question EE} than the general case.  In other words, we will be able to prove it from $\mathrm{C}^*$-algebraic manipulations of our generalized ultraproducts once we know that the Keisler--Shelah characterization of elementary equivalence works for generalized ultrapowers (see Proposition \ref{introprop: ultraproduct EE} below).
	
	Now let us briefly describe the generalized ultraproduct construction in the case of ultrapowers.  It is well-known that ultrafilters on a (discrete) index set $I$ are equivalent to pure states on $\ell^\infty(I)$.  From this point of view, there is no need to restrict $I$ to be a discrete measurable space.  Thus, for a measure space $(\Omega,\mu)$, we can consider a pure state $\mathcal{U}$ on $L^\infty(\Omega,\mu)$ as an ``ultrafilter on a measure space.''  Given a tracial von Neumann algebra $(M,\tau)$, the ultrapower $M^{\cU}$ can then be defined in an analogous way.  Take the tracial state $\cU \circ (\id \otimes \tau)$ on the $L^\infty(\Omega,\mu) \overline{\otimes} M$, and let $M^{\cU}$ be the algebra generated by the GNS construction of this trace, which is a $\mathrm{C}^*$-quotient of $L^\infty(\Omega,\mu) \overline{\otimes} M$.  This $M^{\cU}$ is a special case of the ultrafiber construction described above, namely it is the ultrafiber of $L^\infty(\Omega,\mu) \overline{\otimes} M$ associated to the character $\cU$ on $N = L^\infty(\Omega,\mu)$.  To give a short proof of Corollary \ref{introcol: Ozawa's question EE} we only need the following fact, which allows us to extend the Keisler-Shelah characterization of elementary equivalence to the setting of ultrapowers over measure spaces.
	
	\begin{introprop} \label{introprop: ultraproduct EE}
		Let $\cU$ be a pure state on $L^\infty(\Omega,\mu)$ for some measure space, and let $(M,\tau)$ be a tracial von Neumann algebra.  Then $M^{\cU} \equiv M$.
	\end{introprop}
	
	For Theorem \ref{introthm: EE disintegration}, we generalize this ultrapower construction by replacing $L^\infty(\Omega,\mu)$ inside $L^\infty(\Omega,\mu) \otimes M$ with a general $N \subseteq M$ such that $N$ is contained in the center of $M$ (for instance, $N$ could be $L^\infty(\Omega,\mu)$ inside a direct integral $M$).  One of the key ingredients for all our results is the following analog of {\L}o{\'s}'s theorem in this setting, which is closely related to \cite[Theorem 3.19]{BY13}.  We hope that our proof  will make the result accessible to operator algebraists.
	
	\begin{introthm}[{\L}o{\'s} theorem for ultrafibers] \label{introthm: Los}
		Let $(M,\tau)$ be a tracial von Neumann algebra, and let $N \subseteq Z(M)$.  Let $E: M \to N$ be the trace-preserving conditional expectation.  For each $n$-variate formula $\phi$, let $\phi_E^M$ be the $N$-valued interpretation of $\phi$ on elements of $M$ using $E$ in place of the trace (see \S \ref{subsec: operator valued interpretation}).
		
		Let $\cU$ be a pure state on $N$ and let $M^{/E,\cU} = M / I_{\cU}$ be the ultrafiber (see \S \ref{subsec: ultrafiber}). 
		Let $\overline{x} \in M^n$ and let $\pi_{E,\cU}: M \to M^{/E,\cU}$ be the quotient map.  Then for each formula $\phi$, we have
		\[
		\phi^{M^{/E,\cU}}(\pi_{E,\cU}(\overline{x})) = \cU[\phi_E^M(\overline{x})].
		\]
	\end{introthm}
	
	In \S \ref{sec: saturation}, we point out that, while these generalized ultraproducts are a useful tool, they do not necessarily produce new objects. Indeed, if we assume the continuum hypothesis, then we can show by standard techniques that each ultrafiber of a measurable field $(M_\omega,\tau_\omega)$ can be realized as a discrete ultraproduct of some sequence of elements $(M_{\omega_n},\tau_{\omega_n})$.
	
	It is natural to ask to what extent these results generalize beyond the tracial setting, that is, to von Neumann algebras with faithful normal states.  Theorem \ref{introthm: Los} certainly applies for general metric structures, hence for von Neumann algebras with states as axiomatized in \cite{Dabrowski2019}, \cite[\S 5]{AGHS2025}.  Theorem \ref{introthm: EE disintegration} relies on the fact that the center of ultraproducts is the ultraproduct of the center (see Proposition \ref{prop: center of ultrafiber} and the proof of Proposition \ref{prop: distribution ultrapower}); this will fail for type $\mathrm{III}_0$ von Neumann algebras (see \cite[Fact 8.5]{AGHS2025}).  There may also be difficulty in the type $\mathrm{II}_\infty$ setting due to the lack of axiomatizability of this class \cite[Proposition 8.3]{AGHS2025}.  However, we expect that the result will hold if the algebras and states in the direct integral decomposition are tracial or $\mathrm{III}_\lambda$ for $\lambda > 0$.
	
	We close with a word of motivation on why the notions from continuous model theory, such as formulas, types, and theories, are important for the study of operator algebras (other than for proving Theorem \ref{introthm: EE disintegration}).  Indeed, since complete theories correspond to elementary equivalence classes, which in turn can be characterized in terms of ultraproducts, one might ask what we gain from formalizing the notion of theories through sentences.  One reason is to provide a \emph{topology}; since complete theories can be expressed as characters on a certain real $\mathrm{C}^*$-algebra of sentences, the space of complete theories comes with a natural weak-$*$ topology.  This topology can be characterized by the fact that the theory of an ultraproduct $\prod_{\cU} M_i$ is the ultralimit of the theories of $M_i$.  If we look at things purely in terms of elementary equivalence classes (defined in terms of isomorphic ultrapowers), the topology could only be described in terms of sequences or nets by using the condition $\operatorname{Th}(\prod_{\cU} M_i) = \lim_{\mathcal{U}} \operatorname{Th}(M_I)$, and it is not transparent what the open sets are, why this condition defines a topology, why the space is Hausdorff, etc.  On the other hand, after defining theories properly using sentences, we obtain a compact metrizable space of complete theories, which is actually crucial for proving Theorem \ref{introthm: EE disintegration} in the diffuse case.  The key point is that $\omega \mapsto \operatorname{Th}(M_\omega)$ is a measurable map from the underlying probability space into the space of theories with its Borel $\sigma$-algebra, i.e., $\operatorname{Th}(M_\omega)$ is a random variable in the space of complete theories, and we show that the distribution of $\operatorname{Th}(M_\omega)$ is uniquely determined if $\int_\Omega M_\omega$ is given up to elementary equivalence.
	
	\subsection{Notation}
	
	We assume familiarity with $\mathrm{C}^*$-algebras and von Neumann algebras.  For background, refer to \cite{AP16,Tak79,Blackadar2006,BrownOzawa2008}.  We recall the following terminology and notation:
	\begin{itemize}
		\item A \emph{tracial von Neumann algebra} is a pair $(M,\tau)$ where $M$ is a (necessarily finite) von Neumann algebra and $\tau$ is a faithful normal tracial state.
		\item An \emph{embedding} of tracial von Neumann algebras is a trace-preserving $*$-homomorphism.  Often $M$ and $N$ have a preferred trace in context, and an embedding $M \to N$ will still be assumed to be trace-preserving even if $\tau$ is not written explicitly.
		\item $\norm{\cdot}$ signifies the operator norms on a von Neumann algebra or more generally a $C^\ast$-algebra.
		\item $\norm{\cdot}_{2,\phi}$ denotes the $2$-norm with respect to a given state $\phi$, namely $\norm{x}_{2,\phi} = \phi(x^*x)^{1/2}$.  In particular, when a tracial state is given by context, then $\norm{\cdot}_2$ will denote its associated $2$-norm.
		\item $Z(M)$ denotes the center of a von Neumann algebra $M$.
		\item $M_{\operatorname{sa}}$ denotes the set of self-adjoint elements.
		\item $M_+$ denotes the set of positive elements.
		\item $U(M)$ denotes the set of unitary elements.
		\item $(M)_1$ denotes the unit ball of $M$ with respect to operator norm.
	\end{itemize}
	
	We also briefly recall a definition of direct integrals for tracial von Neumann algebras.  Although von Neumann's work defines direct integrals of von Neumann algebras based on direct integrals of the underlying Hilbert spaces, the direct integral can also be described directly in terms of algebra elements, thanks to the axiomatic characterizations of tracial von Neumann algebras as tracial $\mathrm{C}^*$-algebras whose unit ball is complete in $\norm{\cdot}_2$ \cite[Proposition 2.6.4]{AP16}.  This development is more suited to our purposes and consistent with the general definition for metric structures used in \cite{FG24,BYIT2024}.
	
	\begin{defn}[Measurable fields of tracial von Neumann algebras] \label{def: measurable field}
		Let $(\Omega,\mu)$ be a probability space (i.e.\ a measure space with $\mu(\Omega) = 1$).  Let $(M_\omega,\tau_\omega)$ be a collection of separable tracial von Neumann algebras indexed by $\omega \in \Omega$.  A \emph{section} is a function $f: \Omega \to \bigsqcup_{\omega \in \Omega} M_\omega$ with $f(\omega) \in M_\omega$.  Let $(e_n)_{n \in \N}$ be a sequence of sections such that
		\begin{enumerate}[(1)]
			\item $e_1(\omega) = 1$.
			\item For each $\omega$, $(e_n(\omega))_{n \in \N}$ is dense in $(M_\omega)_1$ with respect to $\norm{\cdot}_{2,\tau_\omega}$.
			\item For each $n$-variable $*$-polynomial $p$, the function $\omega \mapsto \tau_\omega[p(e_1(\omega),\dots,e_n(\omega))]$ is measurable.
		\end{enumerate}
		Then we say that $(M_\omega, (e_n(\omega))_{n \in \N})_{\omega \in \Omega}$ is a \emph{measurable field of (separable) tracial von Neumann algebras}.
	\end{defn}
	
	\begin{defn}
		Continue the same setup as Definition \ref{def: measurable field}.
		\begin{enumerate}
			\item A section $f: \omega \to \bigsqcup_{\omega \in \Omega} M_\omega$ is \emph{measurable} if $\omega \mapsto \tau_\omega(e_n(\omega) f(\omega))$ is measurable for each $n \in \N$.
			\item A section $f: \omega \to \bigsqcup_{\omega \in \Omega} M_\omega$ is \emph{bounded} if $\omega \mapsto \norm{f(\omega)}$ is bounded.
			\item A \emph{simple section} is a section of the form $\sum_{j=1}^n \lambda_j \mathbbm{1}_{S_j}(\omega) e_n(\omega)$ where $\lambda_j \in \C$ and $(S_j)_{j=1}^n$ is a measurable partition of $\Omega$.
		\end{enumerate}
	\end{defn}

	\begin{fact} \label{fact: sections}
		Continue the same setup as Definition \ref{def: measurable field}.
		\begin{enumerate}[(1)]
			\item If $(S_j)_{j=1}^m$ is a measurable partition of $\Omega$ and $p_j$ is a non-commutative $*$-polynomial in $n$ variables, then $\sum_{j=1}^m \mathbbm{1}_{S_j}(\omega) \tau_\omega[p_j(e_1(\omega),\dots,e_n(\omega))]$ is a measurable section, so in particular simple sections are measurable.
			\item If $f$ is a measurable section, then $\omega \mapsto \norm{f(\omega)}_{\tau_\omega,2}$ is measurable since
			\[
			\norm{f(\omega)}_{\tau_\omega,2} = \sup_{n \in \N} \mathbbm{1}_{e_n(\omega) \neq 0} \frac{|\tau_\omega(f(\omega) e_n(\omega))|}{\norm{e_n(\omega)}_{2,\tau_\omega}}.
			\]
			Hence, by polarization, $\omega \mapsto \tau_\omega(f(\omega)^* g(\omega))$ is measurable whenever $f$ and $g$ are measurable sections.
			\item If $f$ is a measurable section, then $\omega \mapsto \norm{f(\omega)}$ is measurable since
			\[
			\norm{f(\omega)} = \sup_{n,m} \mathbbm{1}_{e_n(\omega) \neq 0} \mathbbm{1}_{e_m(\omega) \neq 0} \frac{|\tau_\omega(e_n(\omega)^* f(\omega) e_m(\omega))|}{\norm{e_n(\omega)}_{\tau_\omega,2} \norm{e_m(\omega)}_{\tau_\omega,2}},
			\]
			and $\tau_\omega(e_n(\omega)^* f(\omega) e_m(\omega)) = \tau_\omega(f(\omega) [e_m(\omega) e_n(\omega)^*])$ is measurable by (1) and (2).
			\item If $f_n$ are measurable sections, $f$ is a section, and $\lim_{n \to \infty} \norm{f_n(\omega) - f(\omega)}_{\tau_\omega,2} = 0$ for a.e.\ $\omega$, then $f$ is a measurable section.
			\item Every bounded measurable section is a pointwise-$\norm{\cdot}_2$ limit of simple sections $f_n$ such that $\sup_\omega \norm{f_n(\omega)} \leq \sup_\omega \norm{f(\omega)}$.  Indeed, by rescaling, assume without loss of generality that $\norm{f(\omega)} \leq 1$. Then define $m(\omega,n)$ recursively with $m(\omega,1) = 1$ and
			\[
			m(\omega,n+1) = \begin{cases}
				n+1, & \text{ if } \norm{e_{n+1}(\omega) - f(\omega)}_{2,\tau_\omega} < \norm{e_{m(\omega,n)}(\omega) - f(\omega)}_{2,\tau_\omega} \\
				m(\omega,n), & \text{ else.}
			\end{cases}
			\]
			Then $f_n(\omega) := e_{m(\omega,n)}(\omega)$ is a simple section and $\norm{f_n(\omega) - f(\omega)}_{2,\tau_\omega} \to 0$.
			\item Bounded measurable sections form a $*$-algebra under pointwise multiplication.  To show closure under products, let $f$ and $g$ by bounded measurable sections and let $f_n$ and $g_n$ be approximating sequences of simple sections as in (6).  Then $f_n g_n$ is a measurable section by (1).  We also have $\norm{f_n(\omega) g_n(\omega) - f(\omega) g(\omega)}_{\tau_\omega,2} \to 0$ for each $\omega$, so $fg$ is measurable by (5).  Closure under addition and adjoint are left to the reader.
		\end{enumerate}
	\end{fact}

	\begin{defn}[Direct integrals of tracial von Neumann algebras]
		Let $(\Omega,\mu)$ and $(M_\omega, (e_n(\omega))_{n \in \N})_{\omega \in \Omega}$ be as in Definition \ref{def: measurable field}.  The direct integral is the pair $(M,\tau)$ given as follows:
		\begin{enumerate}
			\item $M$ is the set of bounded measurable sections modulo equality almost everywhere.
			\item The $*$-algebra operations on $M$ are defined pointwise on $\Omega$.
			\item The norm on $M$ is given by $\norm{f} = \operatorname{esssup}_{\omega \in \Omega} \norm{f(\omega)}$.
			\item The trace $\tau$ is given by $\tau(f) = \int_{\Omega} \tau_\omega(f(\omega))\,d\mu(\omega)$.
		\end{enumerate}
	\end{defn}
	
	To see that $(M,\tau)$ is a tracial von Neumann algebra via \cite[Proposition 2.6.4]{AP16}, one first checks it is a $\mathrm{C}^*$-algebra with the norm given by the essential supremum as above, which is immediate once we know that $\omega \mapsto \norm{f(\omega)}$ is measurable for each $f \in M$.  Then  $\tau$ is faithful since $\tau(f^*f) = 0$ implies that $\tau_\omega(f(\omega)^*f(\omega)) = 0$ and hence $f(\omega) = 0$ almost everywhere.  Finally, for completeness of the unit ball in $\norm{\cdot}_{2,\tau}$, suppose that $f_n$ is Cauchy in $\norm{\cdot}_{2,\tau}$.  Take a subsequence $f_{n(k)}$ such that $\sum_k \norm{f_{n(k)} - f_{n(k+1)}}_{2,\tau}^2 < \infty$, and observe that $f_{n(k)}(\omega)$ converges in $\norm{\cdot}_{2,\tau_\omega}$ almost everywhere to some $f(\omega)$.  By Fact \ref{fact: sections} (5), $f$ is measurable.  It also follows from standard measure-theory arguments that $f_n \to f$ in $\norm{\cdot}_{2,\tau}$.
	
	We remark that the set of measurable sections, and hence the direct integral, can depend on the choice of $e_n(\omega)$.  However, as in \cite[Remark 8.13]{BYIT2024}, if two sequences $(e_n(\omega))_{n \in \N}$ and $(e_n'(\omega))_{n \in \N}$ are measurable with respect to each other, then a bounded section $f$ is measurable with respect to $(e_n)_{n \in \N}$ if and only if it is measurable with respect to $(e_n')_{n \in \N}$ in light of Fact \ref{fact: sections} (5) and (4), and hence the direct integrals are the same.
	
	\subsection{Organization}
	
	The rest of the paper is organized as follows:
	\begin{itemize}
		\item \S \ref{sec: ultrafiber} explains the construction of ultrafibers and ultraproducts associated to pure states on a commutative von Neumann algebra, and how to deduce Corollary \ref{introcol: Ozawa's question EE} from Proposition \ref{introprop: ultraproduct EE}.
		\item \S \ref{sec: Los} describes the $N$-valued interpretation of formulas for some $N \subseteq Z(M)$ and proves Theorem \ref{introthm: Los}.
		\item \S \ref{sec: distribution} defines the distribution of theories and shows that it can be recovered when a tracial von Neumann algebra is given up to elementary equivalence, thus proving Theorem \ref{introthm: EE disintegration}.
		\item \S \ref{sec: saturation} shows that under continuum hypothesis, an ultrafiber associated to a separable direct integral will in fact be isomorphic to an ultraproduct over the natural numbers of some of the fibers.
		\item \S \ref{sec: RM} applies the ultrafiber construction to ultraproducts of random matrix algebras.  This gives a new way of making a ``deterministic selection'' of elements in an ultraproduct representing the large-$n$ limit of random matrix models, which sidesteps measurability problems inherent in the study of ultralimits.
	\end{itemize}
	
	\subsection{Acknowledgements}
	
	We thank Ilijas Farah for bringing the problem of direct integrals and model theory to our attention. 
	We thank Srivatsav Kunnawalkam Elayavalli, Jennifer Pi, Ionut Chifan, and David Sherman for comments on a draft of the paper.  We thank the referee for careful reading of the manuscript and helpful suggestions.
	DJ was partially supported by the National Sciences and Engineering Research Council (Canada), grant RGPIN-2017-05650, and the Danish Independent Research Fund, grant 1026-00371B.
	
	\section{Ultrafibers and generalized ultrapowers} \label{sec: ultrafiber}
	
	\subsection{Construction of ultrafibers} \label{subsec: ultrafiber}
	
	As motivation, we recall the relationship between ultrafilters on discrete sets and characters.  A \emph{character} on a commutative (unital) $\mathrm{C}^*$-algebra is a multiplicative linear functional into the complex numbers.  It is a standard fact that characters are equivalent to \emph{pure states}, or extreme points in the spaces of states (positive linear functionals of norm $1$).
	
	For an ultrafilter $\mathcal{U}$ on an index set $I$, the ultralimit $\lim_\mathcal{U}$ defines a character on $l^\infty(I)$. If $\mathcal{U} \neq \mathcal{V}$, then there is a set $A \subseteq I$ with $A \in \mathcal{U}$ but $A \notin \mathcal{V}$, so $\lim_\mathcal{U} \mathbbm{1}_A = 1$ but $\lim_\mathcal{V} \mathbbm{1}_A = 0$, whence the map $\mathcal{U} \mapsto \lim_\mathcal{U}$ is injective. In fact, it is also surjective, for, if $\sigma$ is a character on $l^\infty(I)$, then $\sigma$ must send any projection to either $0$ or $1$. As such, it is easy to verify that
	\begin{equation*}
		\mathcal{U} = \{A \subseteq I: \sigma(1_A) = 1\}
	\end{equation*}
	is an ultrafilter.  For this ultrafilter, clearly $\lim_\mathcal{U} \mathbbm{1}_A = \sigma(\mathbbm{1}_A)$ for any projection $\mathbbm{1}_A$ in $\ell^\infty(I)$. As projections densely span a von Neumann algebra, we see that $\lim_\mathcal{U} = \sigma$. Therefore, we shall use ultrafilters on $I$ and characters on $l^\infty(I)$ interchangeably and use the same symbol to denote both. We shall further generalize to use $\mathcal{U}$ and $\mathcal{V}$ to denote characters on arbitrary commutative von Neumann algebras.
	
	\begin{defn} \label{def: ultrafiber}
		Let $M$ be a finite von Neumann algebra, $N \subseteq Z(M)$ be a subalgebra of its center, $E: M \to N$ be a normal, tracial, faithful conditional expectation, and $\mathcal{U}$ be a character on $N$. The \textit{ultrafiber} of $M$ w.r.t. $E$ and $\mathcal{U}$, denoted by $M^{/E, \mathcal{U}}$, is given by
		\begin{equation*}
			M^{/E, \mathcal{U}} = M/I_{E, \mathcal{U}}
		\end{equation*}
		where $I_{E, \mathcal{U}}$ is the closed ideal given by,
		\begin{equation*}
			I_{E, \mathcal{U}} = \{x \in M: \mathcal{U}(E(x^*x)) = 0\}
		\end{equation*}
		$M^{/E, \mathcal{U}}$ shall always be regarded as equipped with the faithful tracial state $\tau_{E, \mathcal{U}}$ given by,
		\begin{equation*}
			\tau_{E, \mathcal{U}}(x + I_{E, \mathcal{U}}) = \mathcal{U}(E(x))
		\end{equation*}
	\end{defn}

	Consider the case where $M$ is a direct sum $\oplus_{i \in I} M_i$ of tracial von Neumann algebra $(M_i,\tau_i)$ over some index set $I$; suppose $N = \ell^\infty(I)$ and $E: M \to N$ is given by application of the trace on each direct summand.  Suppose $\mathcal{U}$ is the character given by the ultralimit associated to an ultrafilter on $I$, which we shall also denoted by $\mathcal{U}$.  Then it is easy to verify that $M^{/E, \mathcal{U}} = \prod_\mathcal{U} M_i$. Hence, the ultrafiber construction generalizes the classical ultraproducts.  We remark here that although our primary interest is in tracial von Neumann algebras, it is important in Definition \ref{def: ultrafiber} \emph{not} to assume that $M$ admits a faithful tracial state; indeed, if $M$ had a faithful tracial state, then so would $N$, but if $I$ is an uncoutable discrete set, then $\ell^\infty(I)$ does not have a faithful state.

	\begin{prop}\label{prop: ultrafiber-well-defined}
		$(M^{/E, \mathcal{U}}, \tau_{E, \mathcal{U}})$ is a tracial von Neumann algebra. Furthermore, any element of norm 1 in $M^{/E, \mathcal{U}}$ can be lifted to an element of norm 1 in $M$.
	\end{prop}
	
	The proof follows similar lines as the proof that classical ultraproducts are tracial von Neumann algebras in \cite[Proposition 5.4.1]{AP16}.  We recall two ingredients from classical von Neumann algebra theory.
	
	\begin{lemma}\label{neighborhood-of-character}
		Let $N$ be a commutative von Neumann algebra, $\mathcal{U}$ be a character on $N$, $f \in N$. For any $\epsilon > 0$, there exists a projection $p \in N$ s.t. $\mathcal{U}(p) = 1$ and $\|fp - \mathcal{U}(f)p\|_\infty \leq \epsilon$.
	\end{lemma}
	
	\begin{proof} Let $\Omega$ be the spectrum of $N$ and we shall regard $f$ as a continuous function on $\Omega$. Then there exists an open neighborhood $O \subseteq \Omega$ of $\mathcal{U}$ s.t. $|f(t) - \mathcal{U}(f)| \leq \epsilon$ whenever $t \in O$. As clopen sets form a basis of topology for $\Omega$, $O$ can be chosen to be clopen. Letting $p = 1_O$ concludes the proof.
	\end{proof}
	
	\begin{lemma}\label{2-infty-norm}
		Let $M$ be a finite von Neumann algebra, $N \subseteq Z(M)$ be a subalgebra of its center, $E: M \to N$ be a normal, tracial, faithful conditional expectation. Then $\|x\|_{E, 2, \infty} = \|E(x^\ast x)\|_\infty^{1/2}$ defines a norm on $M$ and $(M)_1$ is complete under $\|\cdot\|_{E, 2, \infty}$.
	\end{lemma}
	
	\begin{proof}
		That $\|\cdot\|_{E, 2, \infty}$ defines a norm is an easy exercise. To prove completeness, we let $\{\varphi_i\}_{i \in I}$ be a maximal collection of normal states on $N$ with disjoint supports. Then $\sum_{i \in I} \mathrm{supp}(\varphi_i) = 1$, so the direct sum of the GNS representations of $M$, associated with $\varphi_i \circ E$,
		\begin{equation*}
			\pi = \bigoplus_{i \in I} \pi_{\varphi_i \circ E}
		\end{equation*}
		is a faithful normal representation. Thus, a uniformly dense subset of positive linear functionals on $M$ is of the form,
		\begin{equation*}
			\phi(x) = \sum_{i \in I_0} \varphi_i(E(y_ix))
		\end{equation*}
		where $I_0 \subseteq I$ is a finite subset and $y_i \in M_+$ is supported on $\mathrm{supp}(\varphi_i)$.
		
		Now, let $(x_n) \subseteq (M)_1$ be a Cauchy sequence under $\|\cdot\|_{E, 2, \infty}$. For any finite subset $I_0 \subseteq I$ and $y_i \in M_+$ supported on $\mathrm{supp}(\varphi_i)$, we have,
		\begin{equation*}
			\begin{split}
				\sum_{i \in I_0} \varphi_i(E(y_i(x_n - x_m)^\ast (x_n - x_m))) &\leq \sum_{i \in I_0} \|y_i\|_\infty \varphi_i(E((x_n - x_m)^\ast (x_n - x_m)))\\
				&\leq \sum_{i \in I_0} \|y_i\|_\infty \|x_n - x_m\|_{E, 2, \infty}^2\\
				&\to 0
			\end{split}
		\end{equation*}
		as $n, m \to \infty$. Hence, as $\|x_n\|_\infty \leq 1$ for all $n$, we see that $x_n \to x$ in the strong$^\ast$ topology for some $x \in (M)_1$.
		
		We claim that $\|x_n - x\|_{E, 2, \infty}^2 \to 0$ to conclude the proof. For any $\epsilon > 0$, there exists $N > 0$ s.t. $\|x_n - x_m\|_{E, 2, \infty}^2 \leq \epsilon$ whenever $n, m \geq N$, i.e., $\|E((x_n - x_m)^\ast (x_n - x_m))\|_\infty \leq \epsilon$. As $x_m \to x$ in the strong$^\ast$ topology, $(x_n - x_m)^\ast (x_n - x_m) \to (x_n - x)^\ast (x_n - x)$ in the strong$^\ast$ topology. As $E$ is normal, we have $E((x_n - x_m)^\ast (x_n - x_m)) \to E((x_n - x)^\ast (x_n - x))$ in the strong$^\ast$ topology, whence $\|E((x_n - x)^\ast (x_n - x))\|_\infty \leq \epsilon$, i.e., $\|x_n - x\|_{E, 2, \infty}^2 \leq \epsilon$ whenever $n \geq N$. The concludes the proof.
	\end{proof}

	\begin{proof}[Proof of Proposition \ref{prop: ultrafiber-well-defined}]
		Let $q: M \to M^{/E, \mathcal{U}} = M/I_{E, \mathcal{U}}$ be the natural quotient map. We claim that $q((M)_1)$ is complete under $\|\cdot\|_{\tau_{E, \mathcal{U}}, 2}$. Granted the claim, then as $q((M)_1)$ is contained in $(M^{/E, \mathcal{U}})_1$ and contains all elements of the latter space of operator norm strictly smaller than 1, consequently dense in the latter space under $\|\cdot\|_{\tau_{E, \mathcal{U}}, 2}$, we must have $q((M)_1) = (M^{/E, \mathcal{U}})_1$ and it is complete. Whence, it is a tracial von Neumann algebra with $\tau_{E, \mathcal{U}}$ a faithful, normal, tracial state, by \cite[Proposition 2.6.4]{AP16}.
		
		To prove the claim, we let $(x_n) \subseteq q((M)_1)$ be a Cauchy sequence under $\|\cdot\|_{\tau_{E, \mathcal{U}}, 2}$. By taking a subsequence if necessary, we may assume $\|x_n - x_{n+1}\|_{\tau_{E, \mathcal{U}}, 2} \leq 2^{-(n+1)}$. We shall construct by induction a sequence $(\mathbbm{x_n}) \subseteq (M)_1$, with $q(\mathbbm{x_n}) = x_n$ and $\|\mathbbm{x_n} - \mathbbm{x_{n+1}}\|_{E, 2, \infty} \leq 2^{-n}$. We start with an arbitrary $\mathbbm{x_1} \in (M)_1$ with $q(\mathbbm{x_1}) = x_1$. Now, assume $\mathbbm{x_n}$ up to some $n$ have been constructed, let $\mathbbm{y_{n+1}} \in (M)_1$ be arbitrarily chosen with $q(\mathbbm{y_{n+1}}) = x_{n+1}$. Since,
		\begin{equation*}
			\mathcal{U}(E((\mathbbm{x_n} - \mathbbm{y_{n+1}})^\ast (\mathbbm{x_n} - \mathbbm{y_{n+1}}))) = \|x_n - x_{n+1}\|_{\tau_{E, \mathcal{U}}, 2}^2 \leq 2^{-2(n+1)}
		\end{equation*}
		by Lemma \ref{neighborhood-of-character}, there exists a projection $p \in N$ s.t. $\mathcal{U}(p) = 1$ and,
		\begin{equation*}
			\|E((\mathbbm{x_n} - \mathbbm{y_{n+1}})^\ast (\mathbbm{x_n} - \mathbbm{y_{n+1}}))p\|_\infty \leq 2^{-2n}
		\end{equation*}
		
		Let $\mathbbm{x_{n+1}} = p\mathbbm{y_{n+1}} + (1-p)\mathbbm{x_n}$. Since both $\mathbbm{y_{n+1}}$ and $\mathbbm{x_n}$ have operator norms bounded by $1$, $\mathbbm{x_{n+1}} \in (M)_1$ as well. We have,
		\begin{equation*}
			\begin{split}
				\|\mathbbm{x_n} - \mathbbm{x_{n+1}}\|_{E, 2, \infty}^2 &= \|E((\mathbbm{x_n} - \mathbbm{x_{n+1}})^\ast (\mathbbm{x_n} - \mathbbm{x_{n+1}}))\|_\infty\\
				&= \|E((p\mathbbm{x_n} - p\mathbbm{y_{n+1}})^\ast (p\mathbbm{x_n} - p\mathbbm{y_{n+1}}))\|_\infty\\
				&= \|E((\mathbbm{x_n} - \mathbbm{y_{n+1}})^\ast (\mathbbm{x_n} - \mathbbm{y_{n+1}}))p\|_\infty\\
				&\leq 2^{-2n}
			\end{split}
		\end{equation*}
		
		We also have,
		\begin{equation*}
			\begin{split}
				\mathcal{U}(E((\mathbbm{x_{n+1}} - \mathbbm{y_{n+1}})^\ast (\mathbbm{x_{n+1}} - \mathbbm{y_{n+1}}))) &= \mathcal{U}(E(((1-p)\mathbbm{x_n} - (1-p)\mathbbm{y_{n+1}})^\ast ((1-p)\mathbbm{x_n} - (1-p)\mathbbm{y_{n+1}})))\\
				&= \mathcal{U}((1-p)E((\mathbbm{x_n} - \mathbbm{y_{n+1}})^\ast (\mathbbm{x_n} - \mathbbm{y_{n+1}})))\\
				&= 0
			\end{split}
		\end{equation*}
		where in the final equality we used that $\mathcal{U}(p) = 1$ and therefore $\mathcal{U}(1-p) = 0$. Thus, $\mathbbm{x_{n+1}} - \mathbbm{y_{n+1}} \in I_{E, \mathcal{U}}$ and $q(\mathbbm{x_{n+1}}) = q(\mathbbm{y_{n+1}}) = x_{n+1}$. This concludes the inductive construction.
		
		By Lemma \ref{2-infty-norm}, $\mathbbm{x_n} \to \mathbbm{x}$ for some $\mathbbm{x} \in (M)_1$ in $\|\cdot\|_{E, 2, \infty}$. Let $x = q(\mathbbm{x})$. We shall show that $x_n \to x$ in $\|\cdot\|_{\tau_{E, \mathcal{U}}, 2}$ to conclude the proof of the claim. Indeed,
		\begin{equation*}
			\begin{split}
				\|x_n - x\|_{\tau_{E, \mathcal{U}}, 2}^2 &= \mathcal{U}(E((\mathbbm{x_n} - \mathbbm{x})^\ast (\mathbbm{x_n} - \mathbbm{x})))\\
				&\leq \|E((\mathbbm{x_n} - \mathbbm{x})^\ast (\mathbbm{x_n} - \mathbbm{x}))\|_\infty\\
				&= \|\mathbbm{x_n} - \mathbbm{x}\|_{E, 2, \infty}\\
				&\to 0
			\end{split}
		\end{equation*}
		
		This proves the claim.
	\end{proof}
	
	\begin{defn}[Generalized ultrapowers] \label{def: gen ultrapower}
		Let $(M, \tau)$ be a tracial von Neumann algebra, $A$ be an abelian von Neumann algebra, and $\mathcal{U}$ be a character on $A$. The \textit{generalized ultrapower} of $M$ with respect to $\cU$, denoted by $M^\mathcal{U}$, is given by
		\begin{equation*}
			M^\mathcal{U} = (A \mathbin{\overline{\otimes}} M)^{/E, \mathcal{U}},
		\end{equation*}
		where $E: A \mathbin{\overline{\otimes}} M \to A$ is given by $E = \mathrm{Id} \otimes \tau$. We shall denote the canonical trace on $M^\mathcal{U}$ by $\tau_\mathcal{U}$. The map that sends $x \in M$ to the element represented by $1_A \otimes x$ in $M^\mathcal{U}$ is a trace-preserving embedding and shall be called the \textit{diagonal embedding}.
	\end{defn}
	
	\begin{rmk}
		Consider the case where $A = \ell^\infty(I)$, where $I$ is an index set, and $\mathcal{U}$ is the character given by the ultralimit associated to an ultrafilter on $I$, which we shall also denote by $\mathcal{U}$. Then $M^\mathcal{U}$, as defined here, is easily seen to be the same as the ordinary ultrapower of $M$ with respect to $\mathcal{U}$. The diagonal embedding is also the same map as in the ordinary case. Hence, this construction generalizes the ordinary ultrapowers.
	\end{rmk}
	
	\subsection{Factoriality and centers}
	
	The center of $M^{/E,\cU}$ behaves as one would expect based on the case of ultraproducts over discrete index sets \cite[Corollary 4.3]{FHS13}.  To set up the proof, we first recall the center-valued trace and the Dixmier averaging theorem.
	
	\begin{thm}{(Generalized Dixmier averaging theorem \cite[III.5, Cor.~of Thm.~1]{Dix81})}
		Let $M$ be a finite von Neumann algebra with center $Z(M)$. Then there is a unique normal, faithful, tracial conditional expectation $E: M \to Z(M)$, called the \textit{center-valued trace}. Furthermore, for each $x \in M$, there exists a sequence $x_i$, where each $x_i$ is a finite convex combination of unitary conjugates of $x$, such that $x_i \to E[x]$ in $\norm{\cdot}_\infty$.
	\end{thm}
	
	Now the center of the ultrafiber can be evaluated as follows.
	
	\begin{prop} \label{prop: center of ultrafiber}
		Let $M$ be a finite von Neumann algebra, $N \subseteq Z(M)$ be a subalgebra of its center, $E: M \to N$ be a normal, tracial, faithful conditional expectation, and $\mathcal{U}$ be a character on $N$.  Let $\pi_{E,\cU}: M \to M^{/E,\cU}$ be the canonical projection map.  Then
		\[
		Z(M^{/E,\cU}) = \pi_{E,\cU}(Z(M)) \cong Z(M)^{/E|_{Z(M)},\cU}.
		\]
		In particular, if $N = Z(M)$ and $E$ is the center-valued trace, then $M^{/E, \mathcal{U}}$ is a factor.
	\end{prop}
	
	\begin{proof}
		First, we show that $Z(M^{/E,\cU}) = \pi_{E,\cU}(Z(M))$.  It is immediate that $\pi_{E,\cU}(Z(M)) \subseteq Z(M^{/E,\cU})$ since $\pi_{E,\cU}$ is a $*$-homomorphism.  Conversely, suppose that $x \in Z(M^{/E,\cU})$, and fix $\mathbbm{x} \in M$ with $\pi_{E,\cU}(\mathbbm{x}) = x$.  Let $E_0: M \to Z(M)$ be the center-valued trace.  By the generalized Dixmier averaging theorem, there is a sequence
		\[
		\mathbbm{x}_i = \frac{1}{n_i} \sum_{j=1}^{n_i} \mathbbm{u}_{i,j} \mathbbm{x}_i \mathbbm{u}_{i,j}^*
		\]
		of convex combinations of unitary conjugates of $\mathbbm{x}$ such that $\mathbbm{x}_i \to E_0(\mathbbm{x}) \in Z(M)$ in norm. Since $\pi_{E,\cU}(\mathbbm{u}_{i,j}) \in U(M^{/E, \mathcal{U}})$ and $\pi_{E,\cU}(\mathbbm{x}) = x \in Z(M^{/E, \mathcal{U}})$, we see that $\pi_{E,\cU}(\mathbbm{u}_{i,j}\mathbbm{x}\mathbbm{u}_{i,j}^\ast) = \pi_{E,\cU}(\mathbbm{x})$. Hence, $\pi_{E,\cU}(\mathbbm{x}_i) = \pi_{E,\cU}(\mathbbm{x})$, and so $\pi_{E,\cU}(E_0(\mathbbm{x})) = x$, which implies that $x \in \pi_{E,\cU}(Z(M))$ as desired.
		
		Next, we give an isomorphism $\pi_{E,\cU}(Z(M)) \cong Z(M)^{/E|_{Z(M)},\cU}$.  First, note that $E|_{Z(M)}$ is a faithful normal conditional expectation from $Z(M)$ onto $N$, and so $Z(M)^{/E|_{Z(M)},\cU}$ is well-defined.  Two elements $\mathbbm{x}$ and $\mathbbm{y}$ in $Z(M)$ represent the same element of $Z(M)^{/E|_{Z(M)},\cU}$ if and only if $E((\mathbbm{x} - \mathbbm{y})^*(\mathbbm{x} - \mathbbm{y})) = 0$ if and only if $\mathbbm{x}$ and $\mathbbm{y}$ represent the same element in $M^{/E,\cU}$.  Hence, there is a well-defined injective $*$-homomorphism $Z(M)^{/E|_{Z(M)},\cU} \to M^{/E,\cU}$, and its image is clearly equal to $\pi_{E,\cU}(Z(M))$.
		
		In the case that $N = Z(M)$, then $Z(M)^{/E|_{Z(M)},\cU} \cong \C$ since $\cU$ is a pure state on $Z(M)$, and hence $Z(M^{/E,\cU}) \cong \C$, so $M^{/E,\cU}$ is a factor.
	\end{proof}
	
	\subsection{Proof of Corollary \ref{introcol: Ozawa's question EE} from Proposition \ref{introprop: ultraproduct EE}}\label{tensor-proof-section}
	
	Now we are ready to give a short proof of Corollary \ref{introcol: Ozawa's question EE} showing that if $M$ and $N$ are $\mathrm{II}_1$ factors and $M \otimes L^\infty[0,1] \equiv N \otimes L^\infty[0,1]$, then $M \equiv N$.  We state it here in a slightly more general form.  The only ingredients needed are the Keisler--Shelah theorm for ordinary ultraproducts, the generalized ultraproduct construction, and Proposition \ref{introprop: ultraproduct EE} which shows that $M$ is elementarily equivalent to each of its generalized ultrapowers, which we will be proven later in Proposition \ref{prop: Los ultrapower}.
	
	\begin{prop} \label{prop: general Ozawa's question EE}
		Let $M$ and $N$ be $\mathrm{II}_1$ factors.  Let $A$ and $B$ be diffuse commutative tracial von Neumann algebras.  If $M \overline{\otimes} A \equiv N \overline{\otimes} B$, then $M \equiv N$.
	\end{prop}
	
	\begin{proof}[Proof of Proposition \ref{prop: general Ozawa's question EE} / Corollary \ref{introcol: Ozawa's question EE}]  Since $A \mathbin{\overline{\otimes}} M \equiv B \mathbin{\overline{\otimes}} N$, the Keisler--Shelah theorem implies that $A \mathbin{\overline{\otimes}} M \equiv B \mathbin{\overline{\otimes}} N$ admit isomorphic ultrapowers. In particular, there exists a character $\mathcal{U}$ on some abelian von Neumann algebra $C$ s.t. there exists a trace-preserving isomorphism $\pi: (A \mathbin{\overline{\otimes}} M)^\mathcal{U} \to (B \mathbin{\overline{\otimes}} N)^\mathcal{U}$.  Since $M$ is a factor, $Z(A \mathbin{\overline{\otimes}} M) = A$, so by Proposition \ref{prop: center of ultrafiber}, the center of $(A \mathbin{\overline{\otimes}} M)^\mathcal{U}$ is $A^\mathcal{U}$. Similarly, $Z((B \mathbin{\overline{\otimes}} N)^\mathcal{U}) = B^\mathcal{U}$. Hence, $\pi$ must restrict to a trace-preserving isomorphism between $A^\mathcal{U}$ and $B^\mathcal{U}$. We may thus identify these algebras via $\pi$.
		
		Let $\mathcal{V}$ be any character on $A^\mathcal{U} = B^\mathcal{U}$. Let $E$ denote the center-valued trace of $(A \mathbin{\overline{\otimes}} M)^\mathcal{U}$ as well as the center-valued trace of $(B \mathbin{\overline{\otimes}} N)^\mathcal{U}$. We again identify the two algebras via $\pi$ and note that the $\pi$ must also preserve the center-valued trace. Hence, we must have
		\begin{equation*}
			[(A \mathbin{\overline{\otimes}} M)^\mathcal{U}]^{/E, \mathcal{V}} = [(B \mathbin{\overline{\otimes}} N)^\mathcal{U}]^{/E, \mathcal{V}}
		\end{equation*}
		Note that $(A \mathbin{\overline{\otimes}} M)^\mathcal{U}$ is, by definition, a quotient of $C \mathbin{\overline{\otimes}} A \mathbin{\overline{\otimes}} M$. Let the natural quotient map be $q: C \mathbin{\overline{\otimes}} A \mathbin{\overline{\otimes}} M \to (A \mathbin{\overline{\otimes}} M)^\mathcal{U}$. Then,
		\begin{equation*}
			[(A \mathbin{\overline{\otimes}} M)^\mathcal{U}]^{/E, \mathcal{V}} = (C \mathbin{\overline{\otimes}} A \mathbin{\overline{\otimes}} M)/q^{-1}(I_{E, \mathcal{V}})
		\end{equation*}
		We also have
		\begin{equation*}
			q^{-1}(I_{E, \mathcal{V}}) = \{x \in C \mathbin{\overline{\otimes}} A \mathbin{\overline{\otimes}} M: \mathcal{V}(E(q(x^\ast x))) = 0\}
		\end{equation*}
		It follows from Dixmier averaging as in the proof of Proposition \ref{prop: center of ultrafiber} that $E \circ q = q \circ (\mathrm{Id} \otimes \mathrm{Id} \otimes \tau)$. Since $q$ restricts to a *-homomorphism $C \mathbin{\overline{\otimes}} A \to A^\mathcal{U}$, we see that $\mathcal{V} \circ q|_{C \mathbin{\overline{\otimes}} A}$ is a character on $C \mathbin{\overline{\otimes}} A$. Hence,
		\begin{equation*}
			q^{-1}(I_{E, \mathcal{V}}) = I_{\mathrm{Id} \otimes \mathrm{Id} \otimes \tau, \mathcal{V} \circ q}.
		\end{equation*}
		But this means $[(A \mathbin{\overline{\otimes}} M)^\mathcal{U}]^{/E, \mathcal{V}} = M^{\mathcal{V} \circ q}$. Similarly, $[(B \mathbin{\overline{\otimes}} N)^\mathcal{U}]^{/E, \mathcal{V}}$ is a generalized ultrapower of $N$ as well. Hence, $M$ and $N$ admit generalized ultrapowers that are isomorphic, and so $M \equiv N$ by Proposition \ref{introprop: ultraproduct EE} and the transitivity of elementary equivalence.
	\end{proof}
	
	\section{{\L}o{\'s}'s theorem and elementary equivalence} \label{sec: Los}
	
	\subsection{Model theory background}
	
	To state and prove {\L}o{\'s}'s theorem in this section, we must first review some background from continuous model theory, including formulas, sentences, and theories.  As this section is largely for readers who are less familiar with model theory, we aim to give minimal working definitions of the concepts for tracial von Neumann algebras, rather than completely general and proper definitions.  Unlike \cite{BYBHU08}, we do not aim to describe metric structures in general.  Unlike \cite{FHS2014a}, we do not need to prove that tracial von Neumann algebras fit into the general framework of metric structures, i.e., that they can be axiomatized in some metric language, since this has already been done.  Hence, we can avoid the technicalities of setting up the sorts or domains, and defining the addition, multiplication, and adjoint functions on these domains, that were used in \cite{FHS2014a}, by simply working on the unit ball.  We remark that another axiomatization was given in \cite[Proposition 29.4]{BYIT2024} which works entirely on the unit ball and uses a universal $\mathrm{C}^*$-algebra to define the function symbols.  For further background on continuous model theory, especially for tracial von Neumann algebras, refer to \cite{BYBHU08,FHS2014a,Goldbring2023spectralgap,JekelCoveringEntropy,Hart2023,GH2023}.
	
	Formulas in continuous logic are analogous to formulas in classical first-order logic, that is, expressions composed of basic statements by applying quantifiers (for all, there exists) and connectives (and, or, not, if/then, etc.).  The basic statements depend on the category of objects that we working with.  The \emph{signature} gives the operations and relations for the category, which are used to generate the corresponding \emph{language}, the set of formulas.  For instance, for fields, there are binary operations are $+$ and $\times$, $0$-ary operations (constant symbols) $0$ and $1$, and the relation $=$, so a basic statement in variables $x$ and $y$ might be $x \cdot (y + z) = x + y$.  Then first-order statements for fields are obtained using quantifiers and connectives.  In continuous model theory, we make the following changes:
	\begin{itemize}
		\item Instead of true/false-valued predicates, we consider real-valued predicates.
		\item Instead of logical operations that input and output true/false values, the connectives are continuous functions $\R^k \to \R$.
		\item Instead of $\forall$ and $\exists$, the quantifiers are $\sup$ and $\inf$.
	\end{itemize}
	
	\begin{defn}
		Let $X$ be an infinite set.  \emph{Formulae} in the language of tracial von Neumann algebras with \emph{variables} from the set $X$ are all formal expressions constructed recursively as follows:
		\begin{enumerate}
			\item A \emph{basic formula} is an expression of the form $\varphi = \re\tr(p(x_1,\dots,x_n))$ and $\varphi = \im\tr(p(x_1,\dots,x_n))$ are formulae, for any $*$-polynomial $p$ in variables $x_1,\dots,x_n$ from $X$.
			\item If $\varphi_i$ is a formula for all $1 \leq i \leq n$ and if $f: \mathbb{R}^k \to \mathbb{R}$ is continuous, then $\varphi = f(\varphi_1, \cdots, \varphi_k)$ is a formula.
			\item If $\varphi$ is a formula and $y \in X$, then $\sup_y \varphi$ and $\inf_y \varphi$ are formulas.
		\end{enumerate}
	\end{defn}
	
	\begin{ex}
		An example of a formula would be
		\[
		\varphi = \inf_z (\exp(\sup_y \re \tr(xy - z^3)) - \im \tr((z+x)^4).
		\]
		When we evaluate this formula in a particular tracial von Neumann algebra $M$, the suprema and infima will be taken over the unit ball.
	\end{ex}
	
	In a formula, each occurrence of a variable is either \emph{free} or \emph{bound} to a quantifier $\sup$ or $\inf$.  In basic formulas, all occurrences of variables are free, and when applying connectives the freeness or boundness of each occurrence of a variable is preserved.  When applying $\sup_y$ or $\inf_y$, all free occurrences of the variable $y$ inside the formula become bound to the new quantifier, and the freeness or boundness of the others remains unchanged.  It is possible that some variable occurs multiple times in a formula and one occurrence is bound to a quantifier while another occurrence is not, for instance if we wrote $\tau(x) - \sup_x \tau(xy)$.  However, in this case, we could rename the bound occurrence of $x$ as $z$ and use the equivalent expression $\tau(x) - \sup_z \tau(zy)$.  Thus, we will always assume that our formulas have a distinct variable name associated to each quantifier.  Moreover, we will usually list the free variables as arguments for each formula, so that
	\[
	\varphi(x_1,\dots,x_n)
	\]
	will usually denote a formula whose free variables are $\{x_1,\dots,x_n\}$, and more generally a formula whose free variables are a subset of $\{x_1,\dots,x_n\}$.
	
	Just like non-commutative polynomials, formulas for tracial von Neumann algebra are formal expressions, but they can be evaluated on particular elements in a particular von Neumann algebra $M$ by plugging in operators into the non-commutative polynomials and evaluating the suprema and infima over the unit ball of $M$.  (In fact, one can define formulas with quantifiers $\sup_{x \in (M)_r}$ and $\inf_{x \in (M)_r}$ for each $r > 0$; but by a change of variables $x \mapsto x/r$, this can be transformed into a supremum or infimum over the unit ball, hence here we only consider the unit ball.)
	
	\begin{defn}[Interpretation of formulas] \label{def: interpretation}
		Let $\varphi(x_1,\dots,x_n)$ be a formula for tracial von Neumann algebras.  Let $\mathcal{M} = (M,\tau)$ be a tracial von Neumann algebra.  Then the \emph{interpretation of $\varphi$ in $\mathcal{M}$} is the mapping $\phi^{\mathcal{M}}: M^n \to \R$ described as follows:
		\begin{itemize}
			\item If $\varphi(x_1,\dots,x_n) = \re \tr(p(x_1,\dots,x_n))$ is a basic formula, then $\varphi^{\mathcal{M}}(x_1,\dots,x_n) = \tau(p(x_1,\dots,x_n))$, that is, we evaluate the polynomial on the given $x_1$, \dots, $x_n \in M$.
			\item If $\varphi = f(\varphi_1,\dots,\varphi_k)$ for a connective $f: \R^k \to \R$, where $\varphi_i = \varphi_i(x_1,\dots,x_n)$, then
			\[
			\varphi^{\mathcal{M}}(x_1,\dots,x_n) = f(\varphi^{\mathcal{M}}(x_1,\dots,x_n),\dots,\varphi_k^{\mathcal{M}}(x_1,\dots,x_n)).
			\]
			\item If $\varphi(x_1,\dots,x_n) = \sup_y \psi(x_1,\dots,x_n,y)$, then
			\[
			\varphi^{\mathcal{M}}(x_1,\dots,x_n) = \sup_{y \in M_1} \psi^{\mathcal{M}}(x_1,\dots,x_n,y).
			\]
		\end{itemize}
	\end{defn}
	
	\begin{ex}
		Consider the formula $\varphi(x) = \sup_y \tr((xy - yx)^*(xy - yx))^{1/2}$.  Then
		\[
		\varphi^{\mathcal{M}}(x) = \sup_{y \in (M)_1} \norm{xy - yx}_2.
		\]
		Then $\varphi^{\mathcal{M}}(x) = 0$ if and only if $x \in Z(M)$.
	\end{ex}
	
	A formula in $(x_1,\dots,x_n)$ expresses certain information about $\mathcal{M}$ and $x_1$, \dots, $x_n$.  If there are no free variables, the formula is called a \emph{sentence}, and its interpretation in $\mathcal{M}$ expresses certain first-order information about $\mathcal{M}$ itself.  This information is what we refer to as the \emph{theory of $\mathcal{M}$}.  Note that in discrete logic, the theory of some object $M$ is the set of all sentences in the language that are true for $M$.  Since our sentences take real values, we describe the theory of $M$ as the \emph{functional} that assigns a real value to each sentence.  More precisely, we use the following definition.
	
	\begin{defn} \label{def: EE} ~
		\begin{itemize}
			\item A \emph{sentence} is a formula with no free variables.
			\item Denoting the set of sentences by $\mathcal{S}$, the \emph{theory of $\mathcal{M}$} is the map $\operatorname{Th}(\mathcal{M}): \mathcal{S} \to \R, \varphi \mapsto \varphi^{\mathcal{M}}$. 
			\item Two tracial von Neumann algebras $\mathcal{M}$ and $\mathcal{N}$ are \emph{elementarily equivalent}, written $\mathcal{M} \equiv \mathcal{N}$, if $\operatorname{Th}(\mathcal{M}) = \operatorname{Th}(\mathcal{N})$, or in other words, $\varphi^{\mathcal{M}} = \varphi^{\mathcal{N}}$ for all sentences $\varphi$.
			\item An embedding $\iota: M \to N$ of tracial von Neumann algebras is said to be \emph{elementary} if $\varphi^N(\iota(x_1),\dots,\iota(x_n)) = \varphi^M(x_1,\dots,x_n)$ for all formulas $\varphi$. 
		\end{itemize}
	\end{defn}
	
	In order to better describe the topology for the space of complete theories and to make the above definitions amenable to functional analysis, we will first describe how the set of formulas and the set of sentences in particular can be completed to a real $\mathrm{C}^*$-algebra, and the theory of $\mathcal{M}$ corresponds to a character on this $\mathrm{C}^*$-algebra.
	
	Let $\mathcal{F}_X$ denote the set of formulas with free variables contained in the set $X$.  We observe that since addition and multiplication are continuous functions $\R^2 \to \R$, they can be applied as connectives to formulas.  This implies that $\mathcal{F}_X$ is an algebra over $\R$.  In order to obtain a $\mathrm{C}^*$-algebra, we need a (semi)norm.
	
	\begin{defn}
		For a formula $\varphi(x_1,\dots,x_n)$ for tracial von Neumann algebras, the \emph{uniform seminorm} is given by
		\[
		\norm{\varphi}_u := \sup \{ |\varphi^{\mathcal{M}}(x_1,\dots,x_n)|: \mathcal{M} = (M,\tau), x_1, \dots, x_n \in (M)_1 \},
		\]
		where we take the supremum over all tracial von Neumann algebras $(M,\tau)$ as well as $x_1$, \dots, $x_n$.
	\end{defn}
	
	It is easy to show by induction that $\norm{\varphi}_u < \infty$ for every formula $\varphi$ (this is a special case of Lemma \ref{uniform-bound} below).  One starts with the basic formulas $\tr(p(x_1,\dots,x_n))$ and then shows that boundedness is preserved when applying connectives and quantifiers.  Similarly, it is an exercise to check that for formulas $\varphi$ and $\psi$ and $c \in \R$,
	\begin{align*}
		\norm{\varphi + \psi}_u &\leq \norm{\varphi}_i + \norm{\psi}_u \\
		\norm{c \varphi}_u &= |c| \norm{\varphi}_u \\
		\norm{\varphi \psi}_u & \leq \norm{\varphi}_u \norm{\psi}_u \\
		\norm{\varphi^2}_u &= \norm{\varphi}_u^2.
	\end{align*}

	\begin{defn} \label{def: C* algebra of predicates}
		Fix a countable index set $I$ (for instance, $I = \N$).  Let $\mathcal{F}_X$ be the set of formulas $\varphi$ with variables in $X \sqcup I$ such that the free variables of $\varphi$ are contained in $X$ and the bound variables of $\varphi$ are contained in $I$.  Thus, for instance $\mathcal{F}_{\varnothing}$ is the set of sentences with variables in $I$.  Then let $\mathcal{P}_X$ be the separation-completion of $\mathcal{F}_X$ with respect to the uniform norm.  The elements of $\mathcal{P}_X$ are called \emph{definable predicates} with free variables in $X$.
	\end{defn}
	
	\begin{lemma}
		$\mathcal{P}_X$ is a commutative real $\mathrm{C}^*$-algebra.  Moreover, for each tracial von Neumann algebra $\mathcal{M}$, there is a unique $*$-homomorphism $\Th(\mathcal{M}): \mathcal{P}_{\varnothing} \to \R$ such that $\Th(\mathcal{M})(\varphi) = \varphi^{\mathcal{M}}$ for sentences $\varphi$.
	\end{lemma}
	
	The fact that $\mathcal{P}_X$ is a commutative real $\mathrm{C}^*$-algebra follows from the inqualities mentioned above.  The fact that evaluation $\varphi \mapsto \varphi^{\mathcal{M}}$ gives a $*$-homomorphism on the algebra of sentences is immediate from the definition of the interpretation, specifically how it behaves on connectives.  Since clearly $|\varphi^{\mathcal{M}}| \leq \norm{\varphi}_u$, this mapping extends to a $*$-homomorphism on $\mathcal{P}_{\varnothing}$.  While above we defined $\operatorname{Th}(\mathcal{M})$ as a linear functional on $\mathcal{S}$, it also defines a linear functional on $\mathcal{P}_{\varnothing}$, and we will denote this functional also by $\operatorname{Th}(\mathcal{M})$ since there is little risk of confusion.
	
	\begin{prop}  \label{prop: theories as characters} ~
		\begin{enumerate}
			\item If $\mathcal{M}$ is a tracial von Neumann algebra, then $\operatorname{Th}(\mathcal{M})$ is a character on $\mathcal{P}_{\varnothing}$.
			\item Conversely, every character on $\mathcal{P}_{\varnothing}$ is equal to $\operatorname{Th}(\mathcal{M})$ for some $\mathcal{M}$.
			\item If $\mathcal{M} = \prod_{\mathcal{U}} \mathcal{M}_i$ for some ultrafilter $\mathcal{U}$ on an index set $I$, then $\operatorname{Th}(\mathcal{M}) = \lim_{\mathcal{U}} \operatorname{Th}(\mathcal{M}_i)$ with respect to the weak-$*$ topology.
		\end{enumerate}
	\end{prop}
	
	This proposition is well-known in continuous model theory, so we merely sketch the proof in an expository way.  We already explained why (1) is true.  For point (3), we use {\L}o{\'s}'s theorem, which shows that for all sentences $\varphi$,
	\[
	\varphi^{\mathcal{M}} = \lim_{\mathcal{U}} \varphi^{\mathcal{M}_i}.
	\]
	Because sentences give a $\norm{\cdot}_u$-dense subset of $\mathcal{P}_{\varnothing}$ by construction, it follows that $\varphi^{\mathcal{M}} = \lim_{\mathcal{U}} \varphi^{\mathcal{M}_i}$ for all $\varphi \in \mathcal{P}_{\varnothing}$, and thus $\operatorname{Th}(\mathcal{M}) = \lim_{\cU} \operatorname{Th}(\mathcal{M}_i)$ in the weak-$*$ topology.
	
	Finally, for (2), recall by Gelfand duality that $\mathcal{P}_{\varnothing} \cong C(\mathcal{X};\R)$, where $\mathcal{X}$ is the space of characters on $\mathcal{P}_{\varnothing}$ with the weak-$*$ topology, which is a compact Hausdorff space.  We also have by definition of the uniform norm that
	\[
	\norm{\varphi}_u = \sup_{\mathcal{M}} |\operatorname{Th}(\mathcal{M})[\varphi]|.
	\]
	Generally, if $\mathcal{Y}$ is a subset of a compact Hausdorff space $\mathcal{X}$ and if
	\[
	\norm{f}_{C(\mathcal{X};\R)} = \sup_{y \in \mathcal{Y}} |f(y)| \text{ for } f \in C(\mathcal{X};\R),
	\]
	then $\mathcal{Y}$ must be dense in $\mathcal{X}$.  Therefore, in our case the points in $\mathcal{X}$ of the form $\operatorname{Th}(\mathcal{M})$ must be dense in $\mathcal{X}$.  However, by point (2), they also form a closed set, and hence they are all of $\mathcal{X}$, meaning that every character on $\mathcal{P}_{\varnothing}$ is of the form $\operatorname{Th}(\mathcal{M})$.
	
	The statement above that points of the form $\operatorname{Th}(\mathcal{M})$ give a closed set is essentially a functional-analytic viewpoint on the compactness theorem in continuous logic \cite[Theorem 5.8]{BYBHU08}.  We close by recalling a basic fact about separability.
	
	\begin{lemma} \label{lem: separability}
		Let $X$ be a countable set of variables.  The real $\mathrm{C}^*$-algebra $\mathcal{P}_X$ of formulas in the language of tracial von Neumann algebras with free variables from $X$ is separable.  Hence, the unit ball in the dual with the weak-$*$ topology is metrizable, and the space of characters with the weak-$*$ topology is metrizable.  In particular, the space of theories of tracial von Neumann algebras is compact and metrizable.
	\end{lemma}
	
	\begin{proof}
		Consider \emph{rational polynomial formulas} defined in a similar way to formulas, except that
		\begin{itemize}
			\item The basic formulas are $\re \tr(p(x_1,\dots,x_n))$ and $\im \tr(p(x_1,\dots,x_n))$ where $p$ is a non-commutative $*$-polynomial with coefficients in $\Q[i]$.
			\item The connectives are polynomial functions $\R^k \to \R$ with rational coefficients.
		\end{itemize}
		The set of rational polynomial formulas is countable if the underlying set of variables is countable.  The rational polynomials in variables $X \sqcup I$ are dense in the space of all formulas with respect to $\norm{\cdot}_u$.  To show this, one proceeds by induction on the complexity of the formula.  For the base case, it suffices to approximate the coefficients of each non-commutative $*$-polynomial by elements of $\Q[i]$ and estimate the uniform norm of the difference.  For applying a connective $f$ to formulas $\phi_1$, \dots, $\phi_k$, recall by the Stone-Weierstrass theorem that $f$ can be uniformly approximated by a polynomial on $[-\norm{\varphi_1}_u,\norm{\varphi_1}_u] \times \dots \times [-\norm{\varphi_k}_u, \norm{\varphi_k}_u]$.  The case of applying a quantifier is immediate since
		\[
		\norm{\sup_y \varphi(x_1,\dots,x_n,y) - \sup_y \psi(x_1,\dots,x_n,y)}_u \leq \norm{\phi(x_1,\dots,x_n) - \psi(x_1,\dots,x_n)}_u.
		\]
		Hence, rational polynomial formulas are dense in the space of formulas.  The same holds when we restrict to formulas where the free variables are in $X$ and the bound variables are in $I$.  This shows separability of $\mathcal{P}_X$ since this is the separation-completion of the space of such formulas.  Since $\mathcal{P}_X$ is separable, the weak-$*$ topology on the unit ball in the dual space is metrizable, and hence also the space of states and the space of characters are metrizable.
	\end{proof}
	
	\subsection{Operator-valued interpretations of formulas} \label{subsec: operator valued interpretation}
	
	In order to understand the theory of a direct integral $M = \int_\Omega M_\omega\,d\mu(\omega)$, we must study the fiberwise evaluation
	\[
	\varphi^{M_\omega}(x_1(\omega),\dots,x_n(\omega))
	\]
	for a formula $\varphi$ on $x_1$, \dots, $x_n \in M$, which (we will see) is a measurable function on $\Omega$.  More generally, since we are interested in non-separable von Neumann algebras such as ultraproducts, we want to replace $L^\infty(\Omega)$ by a von Neumann subalgebra $N \subseteq Z(M)$, and still make sense of ``fiberwise evaluation'' of a formula with respect to $N$.  This leads to the definition of $N$-valued interpretations of formulas; we show below in Lemma \ref{direct-integral-case} that this agrees with the fiberwise evaluation in the case that $M$ is a direct integral.
	
	For (2) below, we recall that multivariable continuous function calculus is well-defined on any \emph{commutative} $\mathrm{C}^*$-algebra.  For (3), we recall that every bounded family of self-adjoint elements in a \emph{commutative} von Neumann algebra has a supremum.
	
	\begin{defn}[$N$-valued interpretation of formulas]
		Let $M$ be a finite von Neumann algebra, $N \subseteq Z(M)$ be a subalgebra of its center, and $E: M \to N$ be a normal, faithful, tracial conditional expectation.  For formulas $\varphi$ in the language of tracial von Neumann algebras in variables from a set $X$, we define the \textit{$N$-valued interpretation} $\varphi_E^M(x_1,\dots,x_n)$ as follows, by induction on the complexity of $\varphi$.  Here we write $\overline{x} = (x_1,\dots,x_n)$.
		\begin{enumerate}
			\item If $\varphi = \re\tau(p(\overline{x}))$ for some *-polynomial $p$, then $\varphi^M_E(\overline{x}) = \re(E(p(\overline{x})))$. Similarly, if $\varphi = \im\tau(p(\overline{x}))$ for some *-polynomial $p$, then $\varphi^M_E(\overline{x}) = \im(E(p(\overline{x})))$.
			\item If $\varphi = f(\varphi_1, \cdots, \varphi_k)$ where $f: \mathbb{R}^l \to \mathbb{R}$ is a connective, i.e., a continuous function, then $\varphi^M_E(\overline{x}) = f((\varphi_1)^M_E(\overline{x}), \cdots, (\varphi_n)^M_E(\overline{x}))$, where $f$ is applied in the sense of continuous functional calculus.
			\item If $\varphi = \sup_y \phi(\overline{x}, y)$, then $\varphi^M_E(\overline{x}) = \sup_{y \in (M)_1} \psi^M_E(\overline{x}, y)$. Similarly, if $\varphi = \inf_y \phi(\overline{x}, y)$, then $\varphi^M_E(\overline{x}) = \inf_{y \in (M)_1} \psi^M_E(\overline{x}, y)$.
		\end{enumerate}
	\end{defn}
	
	We remark that in the case that $N = \mathbb{C}$ and $E = \tau$ is a faithful, normal, tracial state, then $\varphi^M_E$ reduces to the usual interpretation $\varphi^M$ (Definition \ref{def: interpretation}).  As in that case, one can show by induction that the $N$-valued interpretation of every formula is bounded on the unit ball $(M)_1^n$, which is used in (3) above to show that the supremum exists in $N$.
	
	\begin{lemma}\label{uniform-bound}
		Let $M$ be a finite von Neumann algebra, $N \subseteq Z(M)$ be a subalgebra of its center, $E: M \to N$ be a normal, faithful, tracial conditional expectation, and $\varphi$ be a formula with $n$ free variables. Then there exists $C > 0$ s.t. $\norm{\varphi^M_E(\overline{x})}_\infty \leq C$ for all $\overline{x} \in (M)_1^n$. Furthermore, $C$ can be chosen independent of $M$, $N$, and $E$.
	\end{lemma}
	
	\begin{proof}
		The case where $\varphi(\overline{x}) = \re \tr(p(\overline{x}))$ or $\im \tr(p(\overline{x}))$ follows by expanding the polynomial $p$ termwise and using the triangle inequality and submultiplicativity of the operator norm.  Next, suppose that $\varphi = f(\varphi_1,\dots,\varphi_k)$ for some continuous $f: \R^k \to \R$.  Since each of the $\varphi_j$'s is bounded by some constant, the output $f(\varphi_1,\dots,\varphi_k)$ will also be bounded, by the spectral mapping property for continuous functional calculus.  Finally, suppose that $\varphi(\overline{x}) = \sup_y \psi(\overline{x},y)$.  Clearly, if $\norm{\psi_E^M(\overline{x},y)} \leq C$ for all $y$, then $\norm{\varphi_E^M(x)} \leq C$ as well.
	\end{proof}
	
	A key ingredient for our version of {\L}o{\'s}'s theorem, as well as for handling other examples and applications, is the fact that when $\varphi = \sup_y \phi(\overline{x}, y)$, there is a single $y$ that realizes the supremum $\sup_{y \in (M)_1} \psi^M_E(\overline{x}, y)$ within error $\epsilon$ uniformly, i.e., in operator norm.  In the case of a direct integral, this means that $y(\omega)$ gets within $\epsilon$ of the supremum on \emph{almost every} fiber simultaneously.  A similar statement is given in \cite[Lemma 3.13]{BY13} for randomizations.
	
	\begin{prop} \label{prop: maximizer}
		Let $M$ be a finite von Neumann algebra, $N \subseteq Z(M)$ be a subalgebra of its center, $E: M \to N$ be a normal, faithful, tracial conditional expectation, and $\varphi = \sup_y \psi(\overline{x}, y)$ where $\phi$ is a formula with $n + 1$ free variables. Then for any $\overline{x} = (x_1, \cdots, x_n)$, tuple of elements of $M$, and $\epsilon > 0$, there exists $y \in (M)_1$ such that 
		\[
		\psi^M_E(\overline{x}, y) \geq \varphi^M_E(\overline{x}) - \epsilon \text{ in } N_{\sa}.
		\]
	\end{prop}
	
	This property emerges from the ability to paste elements of the von Neumann algebra using partitions of unity on the commutative algebra.  Recall that a \textit{projection-valued measure (PVM)} over von Neumann algebra $N$ is a collection $\{e_i\}_{i \in I}$ of pairwise orthogonal projections in $N$ which sum up to 1.  The relationship between projection-valued measures and operator-valued interpretations of formulas is as follows.  This is parallel to the \emph{locality} property for randomizations \cite[Definition 3.12]{BY13}.
	
	\begin{lemma}\label{formula-partition}
		Let $M$ be a finite von Neumann algebra, $N \subseteq Z(M)$ be a subalgebra of its center, $E: M \to N$ be a normal, faithful, tracial conditional expectation, $\varphi$ be a formula with $n$ free variables $\overline{x} = (x_1, \cdots, x_n)$. Let $\{e_j\}_{j \in J}$ be a PVM over $N$ and $x_{i,j} \in (M)_1$ for $i = 1, \dots, n$ and $j \in J$. Then
		\[
		\sum_{j \in J} e_j\varphi^M_E(x_{1,j},\dots,x_{n,j}) = \varphi^M_E\left( \sum_{j \in J} e_j x_{1,j},\dots, \sum_{j \in J} e_j x_{n,j} \right).
		\]
	\end{lemma}
	
	\begin{proof}
		We proceed by induction on complexity of the formula.  First, suppose $\varphi(\overline{x}) = \re \tr(p(\overline{x}))$, and hence $\varphi_E^M(x_{1,j},\dots,x_{n,j}) = \re E[p(x_{1,j},\dots,x_{n,j})]$.  Write $x_i = \sum_{j \in J} e_j x_{i,j}$.  Since $e_j$ is central,
		\begin{align*}
			e_j \re E[p(x_{1,j},\dots,x_{n,j})] &= e_j \re E[p(e_j x_{1,j},\dots,e_j x_{n,j})] \\
			&= e_j \re E[p(e_j x_1,\dots,e_j x_n)] \\
			&= e_j \re E[p(x_1,\dots,x_n)].
		\end{align*}
		Summing over $j$ completes the proof for this case, and the imaginary part is similar.
		
		Next, suppose that $\varphi = f(\varphi_1,\dots,\varphi_k)$ for some continuous $f: \R^k \to \R$, such that $\varphi_j$ satisfies the induction hypothesis.  Let $\overline{x}_j = (x_{1,j},\dots,x_{n,j})$.  Then we have
		\[
		e_j f((\varphi_1)_E^M(\overline{x}_j),\dots,(\varphi_k)_E^M(\overline{x})) = e_j f(e_j (\varphi_1)_E^M(\overline{x}), \dots, e_j (\varphi_k)_E^M(\overline{x}))
		\]
		because the mapping $N \to e_jN$, $y \mapsto e_j y$ is a $*$-homomorphism of commutative $\mathrm{C}^*$-algebras and therefore respects functional calculus.  From here we apply the induction hypothesis to $\varphi_i$ and then argue similarly as in the first case.
		
		Now suppose that $\varphi(\overline{x}) = \sup_y \psi(\overline{x},y)$ where $\psi$ satisfies the induction hypothesis.  First, note that if $(z_i)_{i \in I}$ is a family of self-adjoint operators in $N$ and $e$ is a projection in $N$, then
		\[
		e \sup_i z_i = \sup_i e z_i.
		\]
		Indeed, clearly $e \sup_i z_i$ is an upper bound for $ez_i$ for each $i$, and hence $e \sup_i z_i \geq \sup_i e z_i$.  On the other hand, let $C = \sup_i \norm{z_i}$.  Then clearly $(1 - e)C + \sup_i ez_i$ is an upper bound for each $z_i$, so that $\sup_i z_i \leq (1 - e)C + \sup_i ez_i$, which implies
		\[
		e \sup_i z_i \leq e \sup_i ez_i \leq \sup_i ez_i \leq e \sup_i z_i,
		\]
		which proves the claim.  Again let $x_i = \sum_j e_j x_{i,j}$.  Note that
		\begin{align*}
			e_j \varphi_E^M(x_{1,j},\dots,x_{n,j}) &= e_j \sup_{y \in (M)_1} \psi_E^M(x_{1,j},\dots,x_{n,j},y) \\
			&= \sup_{y \in (M)_1} e_j \psi_E^M(x_{1,j},\dots,x_{n,j},y) \\
			&= \sup_{y \in (M)_1} e_j \psi_E^M(e_j x_{1,j},\dots,e_jx_{n,j}, e_jy) \\
			&= \sup_{y \in (M)_1} e_j \psi_E^M(e_j x_1,\dots,e_j x_n, e_j y),
		\end{align*}
		and by the reverse manipulations as before, this equals $e_j \varphi_E^M(x_1,\dots,x_n)$.  Hence,
		\[
		\sum_{j \in J} e_j \varphi_E^M(x_{1,j},\dots,x_{n,j}) = \sum_{j \in J} e_j \varphi_E^M(x_1,\dots,x_n) = \varphi_E^M(x_1,\dots,x_n)
		\]
		as desired.  The case of an infimum is symmetrical; in fact, one can use the identity
		\[
		\inf_{y \in (M)_1} \varphi_E^M(\overline{x},y) = -\sup_{y \in (M)_1} [-\varphi_E^M(\overline{x},y)]
		\]
		to reduce it to the case of suprema and connectives.
	\end{proof}
	
	The other ingredient that we need for Proposition \ref{prop: maximizer} is a classical fact about suprema in commutative von Neumann algebras, or equivalently, about Stone spaces (which have already made their appearance in model theory of probability spaces \cite[\S 2]{BY13}).  Lemma \ref{partition-of-unity} serves a similar purpose for us as the measure-theoretic fact \cite[Fact 3.30]{BY13} used in \cite[Lemma 3.31]{BY13}.
	
	\begin{lemma}\label{partition-of-unity}
		Let $N$ be a commutative von Neumann algebra, and let $\{z_i\}_{i \in I} \subset N_{sa}$ be a collection of elements which are uniformly bounded in norm. Then for any $\epsilon > 0$, there exists a PVM $\{e_j\}_{j \in J}$ over $N$, and for each $j$ there exists $i_j \in I$, s.t. $\norm{\sup_{i \in I} z_i - \sum_j e_jz_{i_j}} \leq \epsilon$.
	\end{lemma}
	
	\begin{proof}
		Let $\Omega$ be the Gelfand spectrum of $N$, so that $N \cong C(\Omega)$.  We shall regard all $z_i$ as well as $z = \sup_{i \in I} z_i$ as continuous functions on $\Omega$. Then by \cite[Corollary III.1.16]{Tak79}, there exists an open dense set $O \subseteq \Omega$ s.t. for any $t \in O$, $z(t) = \sup_{i \in I} z_i(t)$. We now consider a maximal collection $\{K_j\}_{j \in J}$ of nonempty pairwise disjoint clopen subsets s.t. for each $j$ there exists $i_j \in I$ with $|z(t) - z_{i_j}(t)| < \epsilon$ for all $t \in K_j$. We claim that $\cup K_j$ is dense. Indeed, assume otherwise, then $U = \Omega \setminus \overline{\cup K_j}$ is a nonempty open set, whence so is $O \cap U$. Fix any $t_0 \in O \cap U$, as $z(t_0) = \sup_{i \in I} z_i(t_0)$, there exists some $z_i$ s.t. $|z(t_0) - z_i(t_0)| < \epsilon$. Continuity then implies $|z(t) - z_i(t)| < \epsilon$ in some neighborhood $K$ of $t_0$. As clopen sets form a basis of topology for $\Omega$ (see \cite[Definition III.1.6]{Tak79}), such a neighborhood can be chosen to be clopen and a subset of $U$. But then adding $K$ to $\{K_j\}$ shows the latter collection is not maximal, a contradiction. This shows that $\cup K_j$ is indeed dense. Let $e_j = 1_{K_j}$.  Then $|\sum_j e_j z_{i_j}(t) - z(t)| \leq \epsilon$ on a dense subset of $\Omega$, and hence by continuity $\norm{\sum_j e_j z_{i_j}(t) - z(t)} \leq \epsilon$.
	\end{proof}

	\begin{proof}[Proof of Proposition \ref{prop: maximizer}]
		As $\varphi^M_E(\overline{x}) = \sup_{y \in (M)_1} \psi^M_E(\overline{x}, y)$, by Lemma \ref{partition-of-unity}, there exists a PVM $\{e_j\}_{j \in J}$ over $N$ and elements $y_j \in (M)_1$ such that
		\[
		\sum_j e_j\psi^M_E(\overline{x}, y_j) \geq \varphi^M_E(\overline{x}) - \epsilon.
		\]
		Let $y = \sum_j e_j y_j$. By Lemma \ref{formula-partition}, to the decompositions $x_i = \sum_j e_j x_i$ and $y = \sum_j e_j y_j$, we have
		\[
		\psi^M_E(\overline{x}, y) = \sum_j e_j \psi_E^M(x_1,\dots,x_n, y_j) \geq \varphi_E^M(\overline{x}) - \epsilon,
		\]
		which concludes the proof.
	\end{proof}
	
	We next note that for a tracial von Neumann algebra $M$, the $N$-valued interpretation of a formula $\varphi(x_1,\dots,x_n)$ always defines a uniformly continuous function on $(M)_1^n$ with respect to $2$-norm, which will be useful for various limiting arguments.  This is a generalization of the well-known uniform continuity property for scalar-valued interpretations of formulas \cite[Theorem 3.5]{BYBHU08}, and its proof proceeds in the same way by induction on the complexity of the formula, while also relying on Proposition \ref{prop: maximizer}.
	
	\begin{lemma}\label{uniform-continuity}
		Let $M$ be a finite von Neumann algebra, $N \subseteq Z(M)$ be a subalgebra of its center, $E: M \to N$ be a normal, faithful, tracial conditional expectation.  Let $\tau$ be a state on $N$, so that $\tilde{\tau} = \tau \circ E$ is a tracial state on $M$.  For each formula $\varphi(x_1,\dots,x_n)$, there exists a continuous increasing function $\delta: [0,\infty) \to [0,\infty)$ such that for $\overline{x}$, $\overline{y} \in (M)_1^n$, we have
		\[
		\max_j \norm{x_j - y_j}_2 \leq \delta(\epsilon) \implies \norm{\varphi_E^M(x_1,\dots,x_n) - \varphi_E^M(y_1,\dots,y_n)}_2 \leq \epsilon,
		\]
		where $\norm{\cdot}_2$ denotes the $2$-norm associated to $\tilde{\tau}$.  Furthermore, $\delta$ can be chosen independent of $M$, $N$, $E$, and $\tau$.
	\end{lemma}
	
	\begin{proof}
		We proceed by induction on the complexity of formulas.  First, consider a basic formula $\re \tr(p(x_1,\dots,x_n))$.  By decomposing $p$ into monomials, it suffices to show that $E[x_{i_1} \dots x_{i_k}]$ is uniformly continuous on $(M)_1^n$.  This follows easily from the inequalities $\norm{xy}_2 \leq \min(\norm{x} \norm{y}_2, \norm{x}_2 \norm{y})$ and $\norm{E(x)}_2 \leq \norm{x}_2$.
		
		Next, suppose that $\varphi = f(\varphi_1,\dots,\varphi_k)$ for some $f: \R^k \to \R$ continuous, and formulas $\varphi_1$, \dots, $\varphi_k$ satisfying the conclusion of the lemma.  By Lemma \ref{uniform-bound}, each formula $\varphi_j$ is bounded by some constant $C_j$.  To obtain the conclusion for $\varphi$, it suffices to show that the application of $f$ by functional calculus defines a $\norm{\cdot}_2$-uniformly continuous function
		\[
		((N)_{C_1} \cap N_{\sa}) \times \dots \times ((N)_{C_k} \cap N_{\sa}) \to N_{\sa}.
		\]
		Note that in the case where $f$ is a polynomial, this follows from the same reasoning as in the first step.  Now fix a sequence of polynomials $f_j$ such that $f_j \to f$ uniformly on $[-C_1,C_1] \times \dots \times [-C_k,C_k]$.  The spectral mapping theorem implies uniform convergence of $f_j \to f$ on $N$ (and with bounds independent of the particular choice of $N$).  Moreover, since uniform continuity is preserved under uniform limits, we see that $f$ is $\norm{\cdot}_2$-uniformly continuous.
		
		Next, consider the case where $\varphi(\overline{x}) = \sup_y \psi(\overline{x},y)$.  Let $\delta: [0,\infty) \to [0,\infty)$ be the modulus of continuity associated to $\psi$, and suppose that $\max_j \norm{x_j - x_j'}_2 \leq \delta(\epsilon / 4)$.  By Proposition \ref{prop: maximizer}, there exists $y$ such that
		\[
		\psi_E^M(\overline{x},y) \geq \varphi_E^M(\overline{x}) - \epsilon/4.
		\]
		Furthermore, note that
		\[
		\varphi_E^M(\overline{x}') \geq \psi_E^M(\overline{x}',y).
		\]
		Letting $[\cdot]_+$ denote the positive part of a self-adjoint operator, we have
		\[
		[\varphi_E^M(\overline{x}) - \varphi_E^M(\overline{x}')]_+ \leq [\varphi_E^M(\overline{x}) - \psi_E^M(\overline{x},y)]_+ + [\psi_E^M(\overline{x},y) - \psi_E^M(\overline{x}',y)]_+ + [\psi_E^M(\overline{x}',y) - \varphi_E^M(\overline{x}',y)]_+.
		\]
		On the right-hand side, the first term is bounded by $\epsilon/4$ by our choice of $y$, and the last term is zero by definition of $\varphi$.  Hence,
		\[
		\norm{ [\varphi_E^M(\overline{x}) - \varphi_E^M(\overline{x}')]_+ }_2 \leq \frac{\epsilon}{4} + \norm{\psi_E^M(\overline{x},y) - \psi_E^M(\overline{x}',y)}_2 \leq \frac{\epsilon}{2},
		\]
		where we have applied the uniform continuity condition on $\psi$.  A symmetrical argument shows that
		\[
		\norm{ [\varphi_E^M(\overline{x}') - \varphi_E^M(\overline{x})]_+ }_2 \leq \frac{\epsilon}{2}.
		\]
		Overall, 
		\[
		\norm{ \varphi_E^M(\overline{x}) - \varphi_E^M(\overline{x}') }_2 = \norm{ [\varphi_E^M(\overline{x}) - \varphi_E^M(\overline{x}')]_+ }_2 + \norm{ [\varphi_E^M(\overline{x}') - \varphi_E^M(\overline{x})]_+ }_2 \leq \epsilon,
		\]
		provided that $\max_j \norm{x_j - x_j'}_2 \leq \delta(\epsilon/4)$.  As before, the case of an infimum is symmetrical to the case of a supremum.
	\end{proof}

	Now we show that if $M$ has a direct integral decomposition, then the operator-valued interpretations of formulas coincide with their fiberwise interpretations.
	
	\begin{lemma}\label{direct-integral-case}
		Suppose that $(M,\tau)$ is a tracial von Neumann algebra given as a direct integral $(M,\tau) = \int_\Omega^{\oplus} (M_\omega,\tau_\omega) \,d\mu(\omega)$.  Let $N = L^\infty(\Omega)$ and let $E: M \to L^\infty(\Omega)$ be the faithful, normal, tracial conditional expectation given by $E[x](\omega) = \tau_\omega(x(\omega))$. Let $\varphi$ be a formula in the language of tracial von Neumann algebras, and let $\overline{x} \in (M)_1^n$.  Then $\varphi^M_E(\overline{x}) \in L^\infty(\Omega)$ is (up to a.e.\ equivalence) the function given by $\Omega \ni \omega \mapsto \varphi^{M_\omega}(\overline{x}(\omega))$.  In particular, $\omega \mapsto \varphi^{M_\omega}(\overline{x}(\omega))$ is measurable.
	\end{lemma}
	
	\begin{proof}
		We proceed by induction on the complexity of formulas. The cases for basic formulas and adding connectives are straightforward and left to the reader.  Now consider $\varphi = \sup_y \psi(\overline{x}, y)$ where the claim is already proved for $\psi$.
		
		Let $z(\omega) = \varphi^{M_\omega}(\overline{x}(\omega))$.  First, let us verify that $z$ is measurable (as in \cite[Lemma 1.2]{FG24}).  Let $(e_n)_{n > 0}$ be as in Definition \ref{def: measurable field}, so that $(e_n(\omega))_{n > 0}$ is $\norm{\cdot}_2$-dense in $(M_\omega)_1$ for every $\omega$.  Note that $\varphi^{M_\omega}$ is uniformly continuous on $(M_\omega)_1$ with respect to $\norm{\cdot}_2$ by Lemma \ref{uniform-continuity} in the case $N = \C$.  Therefore, for almost every $\omega$,
		\begin{equation*}
			\varphi^{M_\omega}(\overline{x}(\omega)) = \sup_{y_0 \in (M_\omega)_1} \psi^{M_\omega}(\overline{x}(\omega), y_0) = \sup_{n > 0} \psi^{M_\omega}(\overline{x}(\omega), e_n(\omega)).
		\end{equation*}
		Thus, $z(\omega) = \varphi^{M_\omega}(\overline{x}(\omega))$ is the pointwise supremum of a countable collection of measurable functions, hence it is measurable.
		
		We show that this agrees with the operator-valued supremum by proving two directions.  Given $y \in (M)_1$, writing $y$ as a function $\omega \mapsto y(\omega) \in (M_\omega)_1$ defined almost everywhere, we have by induction hypothesis
		\begin{equation*}
			\begin{split}
				\psi^M_E(\overline{x}, y)(\omega) &= \psi^{M_\omega}(\overline{x}(\omega), y(\omega))\\
				&\leq \sup_{y_0 \in (M_\omega)_1} \psi^{M_\omega}(\overline{x}(\omega), y_0)\\
				&= \varphi^{M_\omega}(\overline{x}(\omega)).
			\end{split}
		\end{equation*}
		Hence, $z$ is an upper bound for $\psi_E^M(\overline{x},y)$ in $N_{\sa}$, and
		\[
		\varphi_E^M(\overline{x}) = \sup_{y \in (M)_1} \psi_E^M(\overline{x},y) \leq z.
		\]
		On the other hand, since the operator supremum in $L^\infty(\Omega)$ agrees with the pointwise supremum almost everywhere,
		\begin{equation*}
			\begin{split}
				z(\omega) =
				\varphi^{M_\omega}(\overline{x}(\omega)) &= \sup_{n > 0} \psi^{M_\omega}(\overline{x}(\omega), e_n(\omega))\\
				&= \left[ \sup_{n > 0} \psi^M_E(\overline{x}, e_n) \right](\omega)\\
				&\leq \left[\sup_{y \in (M)_1} \psi^M_E(\overline{x}, y) \right](\omega)\\
				&= \left[\varphi^M_E(\overline{x})\right](\omega).
			\end{split}
		\end{equation*}
		Thus, $z \leq \varphi_E^M(\overline{x})$.  Thus we have shown $z = \varphi_E^M(\overline{x})$ as desired.
	\end{proof}
	
	\subsection{{\L}o{\'s}'s theorem for ultrafibers and generalized ultraproducts}
	
	The following theorem is an analog of {\L}o{\'s}'s theorem for ultrafibers.
	
	\begin{thm}[{{\L}o{\'s}'s theorem for ultrafibers; cf. \cite[Theorem 3.19]{BY13}}] \label{thm: continuous Los}
		Let $M$ be a finite von Neumann algebra, $N \subseteq Z(M)$ be a subalgebra of its center, $E: M \to N$ be a normal, faithful, tracial conditional expectation.  Let $\mathcal{U}$ be a character on $N$ and let $\pi_{E,\cU}: M \to M^{/E,\cU}$ be the quotient map onto the ultrafiber.  Then for every formula $\varphi$ and every tuple $\overline{x} \in (M)_1^n$, we have
		\[
		\varphi^{M^{/E, \mathcal{U}}}(\pi_{E,\cU}(\overline{x})) = \mathcal{U}(\varphi^M_E(\overline{x})).
		\]
	\end{thm}

	\begin{proof}
		We proceed by induction on the complexity of $\varphi$. For atomic formulas, this holds because for a non-commutative $*$-polynomial $p$,
		\[
		\tau_{E,\cU}(p(\pi_{E,\cU}(x)) = \tau_{E,\cU}(\pi_{E,\cU}(p(x)) = \cU \circ E[p(x)].
		\]
		The case of adding connectives can also be proved easily from the fact that $\mathcal{U}$ is a $*$-homomorphism from $N$ to $\mathbb{C}$ and therefore commutes with continuous functional calculus.
		
		Next, suppose that $\varphi(\overline{x}) = \sup_y \psi(\overline{x},y)$; the case of an infimum, of course, is symmetrical.  We first show that $\varphi^{M^{/E, \mathcal{U}}}(\pi_{E,\cU}(\overline{x})) \leq \mathcal{U}({\varphi^M_E}(\overline{x}))$.  As noted in Proposition \ref{prop: ultrafiber-well-defined}, any element $\hat{y}$ of norm at most $1$ in $M^{/E, \mathcal{U}}$ can be lifted to an element $y$ of norm at most $1$ in $M$.  From the induction hypothesis
		\[
		\psi^{M^{/E, \mathcal{U}}}(\pi_{E,\cU}(\overline{x}),\hat{y}) = \cU(\psi_E^M(\overline{x},y)) \leq \cU(\varphi_E^M(\overline{x})),
		\]
		and since $\hat{y}$ was arbitrary, $\varphi^{M^{/E,\cU}}(\pi_{E,\cU}(\overline{x})) \leq \cU(\varphi_E^M(\overline{x}))$.  For the other direction, fix $\epsilon > 0$.  Then by Proposition \ref{prop: maximizer}, there exists $y \in (M)_1$ s.t. $\psi^M_E(\overline{x}, y) \geq \varphi^M_E(\overline{x}) - \epsilon$. Then by induction hypothesis,
		\begin{equation*}
			\begin{split}
				\mathcal{U}({\varphi^M_E}(\overline{x})) &\leq \mathcal{U}({\psi^M_E}(\overline{x}, y)) + \epsilon\\
				&= \psi^{M^{/E, \mathcal{U}}}(\pi_{E,\cU}(\overline{x}), \pi_{E,\cU}(y)) + \epsilon\\
				&\leq \varphi^{M^{/E, \mathcal{U}}}(\pi_{E,\cU}(\overline{x})) + \epsilon.
			\end{split}
		\end{equation*}
		Since $\epsilon$ was arbitrary, $\mathcal{U}({\varphi^M_E}(\overline{x})) \leq \varphi^{M^{/E, \mathcal{U}}}(\pi_{E,\cU}(\overline{x}))$, so the proof is complete.
	\end{proof}
	
	\begin{rmk} \label{rem: boundedness}
		We showed earlier that for each formula $\varphi$, there is a universal upper bound for $\varphi_E^M$.  Now in light of Theorem \ref{thm: continuous Los}, we know that for every pure state $\cU$ on $N$, we have
		\[
		|\cU(\varphi_E^M(\overline{x}))| = |\varphi^{M^{/E,\cU}}(\pi_{E,\cU}(\overline{x}))| \leq \norm{\varphi}_u.
		\]
		Therefore,
		\[
		\norm{\varphi_E^M} \leq \norm{\varphi}_u,
		\]
		and it follows that the evaluation map $\varphi \mapsto \varphi_E^M$ extends to the $\mathrm{C}^*$-algebra $\mathcal{P}_X$ of definable predicates from Definition \ref{def: C* algebra of predicates}.
	\end{rmk}
	
	Furthermore, in the special case of ultra\emph{powers} over measure spaces, we have the following, which in particular proves Proposition \ref{introprop: ultraproduct EE}.
	
	\begin{prop} \label{prop: Los ultrapower}
		Let $A$ be a commutative von Neumann algebra and let $(M,\tau)$ be a tracial von Neumann algebra.  Let $\tilde{M} = A \overline{\otimes} M$ and let $\tilde{E} = \id \otimes \tau: \tilde{M} \to A$.  Let $\cU$ be a pure state on $A$, so that $\tilde{M}^{/E,\cU} = M^{\cU}$.  Then for $\overline{x} \in (M)_1^n$ and formulas $\varphi$, we have
		\[
		\varphi^{M^{\cU}}(\pi_{\tilde{E},\cU}(1_A \otimes \overline{x})) = \varphi^M(\overline{x}).
		\]
		In particular, for every sentence $\varphi$, we have $\varphi^{M^{\cU}} = \varphi^M$, so that $M^{\cU} \equiv M$.
	\end{prop}
	
	The first claim says that the mapping $M \to M^{\cU}$ given by $x \mapsto \pi_{E,\cU}(1 \otimes x)$ is an elementary embedding (see \ref{def: EE}).  The second claim about sentences is the special case where $\varphi$ has no free variables (or $n = 0$).  Note that by Theorem \ref{thm: continuous Los} the first claim can be rewritten as
	\[
	\cU \circ \varphi_{\tilde{E}}^{\tilde{M}}(1_A \otimes \overline{x}) = \varphi^M(\overline{x}).
	\]
	Hence, since we want to prove the proposition for \emph{all} pure states, the claim is equivalent to
	\[
	\varphi_{\tilde{E}}^{\tilde{M}}(1_A \otimes \overline{x}) = 1_A \varphi^M(\overline{x}) \text{ in } A.
	\]
	We will prove this in greater generality, allowing $\tau: M \to \C$ to be replaced by a conditional expectation onto a central subalgebra.
	
	\begin{prop}\label{prop: tensor formula}
		Let $M$ be a finite von Neumann algebra, $N \subseteq Z(M)$ be a subalgebra of its center, $E: M \to N$ be a normal, faithful, tracial conditional expectation, and $A$ be an abelian von Neumann algebra. Let $\tilde{M} = A \mathbin{\overline{\otimes}} M$ and $\tilde{E}: \tilde{M} \to A \mathbin{\overline{\otimes}} N$ be given by $\tilde{E} = \mathrm{Id} \otimes E$. Then for any sentence $\varphi$, $\varphi^{\tilde{M}}_{\tilde{E}} = 1_A \otimes \varphi^M_E$.
	\end{prop}
	
	For the sake of induction, we will actually prove an even more general statement.  Proposition \ref{prop: tensor formula} corresponds to the case of Proposition \ref{prop: discretization} where the index set $J$ is a singleton and the projection-valued measure is simply $1_A$.
	
	\begin{prop} \label{prop: discretization}
		Let $M$ be a finite von Neumann algebra, $N \subseteq Z(M)$ be a subalgebra of its center, $E: M \to N$ be a normal, faithful, tracial conditional expectation, and $A$ be an abelian von Neumann algebra. Let $\tilde{M} = A \mathbin{\overline{\otimes}} M$ and $\tilde{E}: \tilde{M} \to A \overline{\otimes} N$ be given by $\tilde{E} = \id \otimes E$. Let $\varphi$ be a formula with $n$ free variables. Let $x_i = \sum_{j \in J} e_j \otimes m_{i,j}$ for each $i \leq n$, where $\{e_j\}_{j \in J}$ is a PVM over $A$ and $m_{i,j} \in (M)_1$ for all $j \in J$.  Also, write $\overline{m}_j = (m_{1,j},\dots,m_{n,j})$. Then
		\[
		\varphi^{\tilde{M}}_{\tilde{E}}(\overline{x}) = \sum_{j \in J} e_j \otimes \varphi^M_E(\overline{m}_j).
		\]
	\end{prop}
	
	\begin{proof}
		We proceed by induction on the complexity of the formula.  As in the proof of Lemma \ref{formula-partition}, the cases for atomic formulae and adding connectives follows from the fact that $\tilde{M} \ni x \mapsto (e_j \otimes 1)x \in (e_j \otimes 1)\tilde{M}$ is a *-homomorphism.  As before, the $\sup$ and $\inf$ cases are symmetrical, and so suppose that $\varphi(\overline{x}) = \sup_y \psi(\overline{x},y)$ where $\psi$ satisfies the induction hypothesis.
		
		Let $e_j$, $m_{i,j}$, and $x_i$ be as in the theorem statement.  Let $\epsilon > 0$.  For each $j$, by Proposition \ref{prop: maximizer}, there exists $n_j$ such that
		\[
		\psi_E^M(\overline{m}_j,n_j) \geq \varphi_E^M(\overline{m}_j) - \epsilon.
		\]
		Let $y = \sum_{j \in J} e_j \otimes n_j$.  Applying the induction hypothesis to $\psi$, we obtain
		\[
		\varphi_{\tilde{E}}^{\tilde{M}}(\overline{x}) \geq \psi_{\tilde{E}}^{\tilde{M}}(\overline{x},y) = \sum_{j \in J} e_j \otimes \psi_E^M(\overline{m}_j, n_j) \geq \sum_{j \in J} e_j \otimes (\varphi_E^M(\overline{m}_j) - \epsilon) \geq \sum_{j \in J} e_j \otimes \varphi_E^M(\overline{m}_j) - \epsilon.
		\]
		Therefore, $\varphi_{\tilde{E}}^{\tilde{M}}(\overline{x}) \geq \sum_{j \in J} e_j \otimes \varphi_E^M(\overline{m}_j)$.
		
		In order to prove the reverse inequality, we start by considering $y$ in certain a dense subset of $(\tilde{M})_1$ with respect to the strong-$*$ operator topology (SOT-$*$).  Let
		\[
		S = \left\{\sum_{k \in K} p_k \otimes n_k, \{p_k\}_{k \in K}\textrm{ is a PVM over }A, n_k \in (M)_1 \right\}.
		\]
		By the construction of tensor products and by the Kaplansky density theorem, $S$ is dense in $(\tilde{M})_1$.  Suppose that $y = \sum_{k \in K} p_k \otimes n_k$ is in $S$.  Then by the induction hypothesis applied to the partition $(e_j p_k)_{j \in J, k \in K}$, we have
		\[
		\psi_{\tilde{E}}^{\tilde{M}}(\overline{x},y) = \sum_{j \in J, k \in K} e_j p_k \otimes \psi_E^M(\overline{m}_j,n_k) \leq \sum_{j \in J, k \in K} e_j p_k \otimes \varphi_E^M(\overline{m}_j) = \sum_{j \in J} e_j \otimes \varphi_E^M(\overline{m}_j).
		\]
		
		Next, suppose that $y$ is a general element in $(\tilde{M})_1$ and write $y$ as the limit of a net $y_i \in S$.  Since $y_i \to y$ in SOT-$*$, we have that for every normal state $\tau$ on $A \overline{\otimes} N$, we have $\norm{y_\ell - y}_{2,\tau \circ \tilde{E}} \to 0$.  By Lemma \ref{uniform-continuity}, this implies that
		\[
		\norm{\psi_{\tilde{E}}^{\tilde{M}}(\overline{x},y_\ell) - \psi_{\tilde{E}}^{\tilde{M}}(\overline{x},y) }_{2, \tau \circ \tilde{E}} \to 0.
		\]
		Therefore, since this holds for every normal state, we have $\psi_{\tilde{E}}^{\tilde{M}}(\overline{x},y_\ell) \to \psi_{\tilde{E}}^{\tilde{M}}(\overline{x},y)$.  The preceding argument showed that 
		\[
		\psi_{\tilde{E}}^{\tilde{M}}(\overline{x},y_\ell) \leq \sum_{j \in J} e_j \otimes \varphi_E^M(\overline{m}_j),
		\]
		and therefore
		\[
		\psi_{\tilde{E}}^{\tilde{M}}(\overline{x},y) \leq \sum_{j \in J} e_j \otimes \varphi_E^M(\overline{m}_j).
		\]
		Since $y \in (\tilde{M})_1$ was arbitrary, $\varphi_{\tilde{E}}^{\tilde{M}}(\overline{x}) \leq \sum_{j \in J} e_j \otimes \varphi_E^M(\overline{m}_j)$.  We have now shown both inequalities so the proof is complete.
	\end{proof}

	This concludes the proof of Proposition \ref{prop: Los ultrapower} and so of Proposition \ref{introprop: ultraproduct EE}.  This proposition allows the Keisler--Shelah characterization of elementary equivalence to be extended to ultrafilters on measure spaces.  We recall the Keisler--Shelah theorem for ultraproducts on discrete index sets here.
	
	\begin{thm}[{Keisler--Shelah, see \cite[Theorem 5.7]{BYBHU08}, \cite[Theorem 2.1(2)]{FHS2014b}}] \label{thm: KS theorem}
		$(M, \tau_M) \equiv (N, \tau_N)$ iff there exists an index set $I$ and an ultrafilter $\mathcal{U}$ on $I$ s.t. $(M,\tau_M)^\mathcal{U} \simeq (N,\tau_N)^\mathcal{U}$.
	\end{thm}
	
	We now have the following corollary.
	
	\begin{col}\label{k-s-l}
		Let $(M, \tau_M)$ and $(N, \tau_N)$ be two tracial von Neumann algebras, then TFAE,
		\begin{enumerate}
			\item $(M, \tau_M) \equiv (N, \tau_N)$;
			\item There exists a character $\mathcal{U}$ on a commutative von Neumann algebra $A$ s.t. $(M,\tau_M)^\mathcal{U} \simeq (N,\tau_N)^\mathcal{U}$.
		\end{enumerate}
	\end{col}
	
	\begin{proof}
		(1) $\implies$ (2).  If $(M,\tau_M) \equiv (N,\tau_N)$, then the Keisler-Shelah theorem implies that $M^{\cU} \cong N^{\cU}$ for some ultrafilter $\cU$ on a discrete index set (i.e., a pure state on $\ell^\infty(I)$ for some index set).
		
		(2) $\implies$ (1).  This follows from Proposition \ref{prop: Los ultrapower} and the transitivity of elementary equivalence.
	\end{proof}
	
	\section{Disintegration for elementary equivalence} \label{sec: distribution}
	
	The goal of this section is to prove Theorem \ref{introthm: EE disintegration}.  It will become clear in the proof that the main case of interest is when the measure spaces are diffuse.  In this case, we will see that the right perspective on this result is to look at the distribution of the theory $\mathrm{Th}(M_\omega)$ as a random variable with values in the space of theories of tracial factors (see \cite[Definition 5.2]{FG24}).  The crux of the proof is to show that if $M \equiv N$, then the distribution of theories for $M$ and $N$ coincide.
	
	\subsection{The distribution of theories} \label{subsec: distribution}
	
	Suppose that $M$ is a tracial von Neumann algebra, $N \subseteq Z(M)$ is a von Neumann subalgebra, and $E: M \to N$ is the unique trace-preserving expectation.  If $\varphi$ is a sentence in the language of tracial von Neumann algebras, then $\varphi_E^M$ is an element of $N$.  Since $(N,\tau|_N)$ is a commutative tracial von Neumann algebra, it is isomorphic to $L^\infty(\Omega,\mu)$ for some complete probability measure space.  Thus, $\varphi_E^M$ can be regarded as a real random variable on $\Omega$.  Hence, it makes sense to speak of its probability distribution, which is its spectral measure on $\R$ with respect to the tracial state.  More generally, we can consider the joint probability distribution of $((\varphi_1)_E^M,\dots, (\varphi_k)_E^M)$ for several sentences $\varphi_1$, \dots, $\varphi_k$, which gives a probability measure on $\R^k$.  As a concrete example, recall that in the case that $M$ is a direct integral over $N = L^\infty(\Omega,\mu)$, then $\varphi_E^M$ coincides with the measurable function $\omega \mapsto \varphi^{M_\omega}$.  Thus, the elements $((\varphi_1)_E^M,\dots, (\varphi_k)_E^M)$ are simply random variables on the underlying probability space for the direct integral.
	
	We actually want to consider the joint distribution of all sentences together, or equivalently, instead of looking at random variables in $\R^k$, we want to look at random variables taking values in the space of theories of tracial von Neumann algebras.  Recall from Definition \ref{def: C* algebra of predicates} that $\mathcal{P}_\varnothing$ is the real $\mathrm{C}^*$-algebra obtained as the separation-completion of the algebra of sentences with respect to $\norm{\cdot}_u$.  We showed in Proposition \ref{prop: theories as characters} that theories of tracial von Neumann algebras are in bijection with characters on $\mathcal{P}_{\varnothing}$.  Let $\Spec(\mathcal{P}_{\varnothing})$ be the space of characters, i.e., the Gelfand spectrum, of $\mathcal{P}_{\varnothing}$, equipped with the weak-$*$ topology, so that $\mathcal{P}_{\varnothing} \cong C(\Spec(\mathcal{P}_{\varnothing});\R)$.  Given $N \subseteq M$ as above, there is a $*$-homomorphism $\ev_E^M: \mathcal{P}_{\varnothing} \to N$ given by $\varphi \mapsto \varphi_E^M$ for sentences (see Remark \ref{rem: boundedness}).  Hence, $\tau_N \circ \ev_E^M$ is a state on $\mathcal{P}_{\varnothing}$, or equivalently a probability measure on $\Spec(\mathcal{P}_{\varnothing})$, which is the distribution that we seek.
	
	\begin{defn}[Distribution of theories]
		We denote by $\mathbb{T}_{\tr} = \Spec(\mathcal{P}_{\varnothing})$ the space of theories of tracial von Neumann algebras with the weak-$*$ topology.  Given a tracial von Neumann algebra $(M,\tau)$ and $N \subseteq Z(M)$, let $E: M \to N$ be the unique trace-preserving conditional expectation.  Let $\ev_E^M: \mathcal{P}_{\varnothing} \to N$ be the unique $*$-homomorphism given by $\varphi \mapsto \varphi_E^M$ on sentences (see Remark \ref{rem: boundedness}).  The \emph{distribution of theories for $M$ over $N$} is the probability measure on $\mathbb{T}_{\tr}$ corresponding to the state $\tau|_N \circ \ev_E^M$ on $\mathcal{P}_{\varnothing}$.
		
		In the case that $N = Z(M)$ and $E$ is the center-valued trace, then we call this probability measure simply the \emph{distribution of theories for $M$}.
	\end{defn}
	
	A more measure-theoretic description of the distribution of theories can be found in Farah and Ghasemi's paper \cite[Definition 5.2]{FG24}.  As one would expect, in the case of direct integrals this can be formulated more concretely.
	
	\begin{lemma} \label{lem: direct integral distribution of theories}
		Suppose $(M,\tau) = \int_{\Omega}^{\oplus} (M_\omega,\tau_\omega)$ is a direct integral of separable tracial von Neumann algebras.  Then $\omega \mapsto \Th(M_\omega,\tau_\omega)$ is a random variable taking values in $\mathbb{T}_{\tr}$, and the distribution of theories for $M$ over $N$ is the distribution of this random variable.
	\end{lemma}
	
	\begin{proof}
		First, to see why $\Th(M_\omega,\tau_\omega)$ is a measurable function, it suffices to show that each open set is measurable.  Recall that $\mathbb{T}_{\tr}$ is separable and metrizable, and the topology is generated by the continuous functions given by elements of $\mathcal{P}_{\varnothing}$, so it suffices to check that $\omega \mapsto \Th(M_\omega,\tau_\omega)[\varphi]$ is measurable for each $\varphi \in \mathcal{P}_{\varnothing}$.  Since formulas give a dense subset of $\mathcal{P}_{\varnothing}$, it suffices to check measurability for formulas, which follows Lemma \ref{direct-integral-case}.
		
		Next, to show that the distribution of $\omega \mapsto \Th(M_\omega,\tau_\omega)$ agrees with the abstract distribution of theories, it suffices to check that both measures agree when integrated against any continuous function on $\mathbb{T}_{\tr}$.  By Lemma \ref{direct-integral-case}, for a sentence $\varphi$, we have
		\[
		\tau_N(\varphi_E^M) = \int_{\Omega} \varphi^{M_\omega}\,d\mu(\omega) = \int_{\Omega} \Th(M_\omega,\tau_\omega)[\varphi]\,d\mu(\omega).
		\]
		By density this extends to all functions in $\mathcal{P}_{\varnothing} = C(\mathbb{T}_{\tr};\R)$, so the claim is proved.
	\end{proof}

	In order to prove that the distribution of theories is preserved by elementary equivalence, we first show the following.  The proof can be thought of as a generalization of the method for Corollary \ref{introcol: Ozawa's question EE}.
	
	\begin{prop} \label{prop: distribution ultrapower}
		Let $(M, \tau)$ be a tracial von Neumann algebra, and let $N$ be a von Neumann subalgebra of its center.  Let $\mathcal{U}$ be a character on a commutative von Neumann algebra $A$, and consider the generalized ultrapowers $N^{\cU} \subseteq M^{\cU}$.  Then the distribution of theories for $M$ over $N$ is the same as the distribution of theories for $M^{\cU}$ over $N^{\cU}$. In particular, the distribution of theories for $M$ over $Z(M)$ is the same as that for $M^{\cU}$ over $Z(M^{\cU})$.
	\end{prop}
	
	\begin{proof}
		Let $q$ denote the natural quotient map $q: A \overline{\otimes} M \to M^\mathcal{U}$.  Note that the trace $\cU \circ (\id_A \otimes \tau)$ on $A \overline{\otimes} M$ restricts on $A \overline{\otimes} N$ to the trace $\cU \circ (\id_A \otimes \tau|_N)$, and hence we have a natural trace-preserving embedding $N^{\cU} \to M^{\cU}$.  Furthermore, letting $\tilde{E}: M^{\cU} \to N^{\cU}$ be the unique trace-preserving conditional expectation, we have
		\[
		\tilde{E} \circ q = q \circ (\id_A \otimes E).
		\]
		Indeed, for each $x \in A \overline{\otimes} M$, $(\id_A \otimes E)(x)$ is an element in $A \overline{\otimes} N$ satisfying
		\[
		(\id_A \otimes \tau)((\id_A \otimes E)(x) y) = (\id_A \otimes \tau)(xy) \text{ for } y \in A \overline{\otimes} N.
		\]
		By applying $\cU$ to both sides, we see that $q((\id_A \otimes E)(x))$ is an element in $N^{\cU}$ with the same inner product as $q(x)$ with every element of $N^{\cU}$.
		
		Let $\mathcal{V}$ be a character on $N^{\cU}$. We claim that $(M^\mathcal{U})^{/\tilde{E}, \mathcal{V}} \cong (A \mathbin{\overline{\otimes}} M)^{/\id_A \otimes E, \mathcal{V} \circ q}$.  Indeed, $(A \overline{\otimes} M)^{/\id_A \otimes E,\mathcal{V} \circ q}$ is obtained by quotienting $A \overline{\otimes} M$ by the ideal $\{x: \mathcal{V} \circ q \circ (\id_A \otimes E)(x^*x) = 0\}$.  Meanwhile, $(M^\mathcal{U})^{/\tilde{E}, \mathcal{V}}$ is obtained by quotienting out the ideal $\{x: \cU \circ (\id_A \otimes \tau)(x^*x) = 0\}$, and then further quotienting by $\{y \in M^{\cU}: \mathcal{V} \circ \tilde{E}(y^*y) = 0\}$.
		But by the first paragraph, $\mathcal{V} \circ \tilde{E} \circ q = \mathcal{V} \circ q \circ (\id_A \otimes E)$, and hence these two quotients are the same.
		
		By Theorem \ref{thm: continuous Los}, for every sentence $\varphi$,
		\[
		\mathcal{V}\left[\varphi_{\tilde{E}}^{M^{\cU}} \right] = \varphi^{(M^\mathcal{U})^{/\tilde{E}, \mathcal{V}}} = \varphi^{(A \mathbin{\overline{\otimes}} M)^{/\id_A \otimes E, \mathcal{V} \circ q}} = \mathcal{V} \circ q \left[ \varphi_{\id_A \otimes E}^{A \overline{\otimes} M} \right].
		\]
		Since $\cV$ was an arbitrary pure state,
		\[
		\varphi_{\tilde{E}}^{M^{\cU}} = q\left[ \varphi_{\id_A \otimes E}^{A \overline{\otimes} M} \right]
		\]
		By Proposition \ref{prop: tensor formula},
		\[
		\varphi_{\id_A \otimes E}^{A \overline{\otimes} M} = 1_A \otimes \varphi_E^M.
		\]
		Therefore,
		\[
		\varphi_{\tilde{E}}^{M^{\cU}} = q \left[ 1_A \otimes \varphi_E^M \right],
		\]
		or in other words, $\varphi_{\tilde{E}}^{M^{\cU}}$ equals the image of $\varphi_E^M$ under the diagonal embedding $N \to N^{\cU}$.  Since the diagonal embedding is trace-preserving, we have
		\[
		\tau_{N^{\cU}}[\varphi_{\tilde{E}}^{M^{\cU}}] = \tau_N[\varphi_E^M].
		\]
		Since this holds for all sentences $\varphi$, we have proved that the distributions of theories agree.
		
		Next, consider the second claim regarding the case where $N = Z(M)$.  By Proposition \ref{prop: center of ultrafiber}, $Z(M^{\cU}) = Z(M)^{\cU}$, so this follows from the first claim.
	\end{proof}
	
	\begin{col} \label{col: same distribution of theories}
		Let $(M, \tau_M)$ and $(N, \tau_N)$ be elementarily equivalent tracial von Neumann algebras. Then their distributions of theories coincide.
	\end{col}
	
	\begin{proof}
		By the Keisler-Shelah theorem (Theorem \ref{thm: KS theorem}), there exists some ultrafilter such that $M^{\cU} \cong N^{\cU}$, which of course implies that $M^{\cU}$ and $N^{\cU}$ have the same distribution of theories.  By Proposition \ref{prop: distribution ultrapower}, $M$ and $M^{\cU}$ have the same distribution of theories, and $N$ and $N^{\cU}$ has the same distribution of theories.  Hence, $M$ and $N$ have the same distribution of theories.
	\end{proof}
	
	\subsection{Proof of Theorem \ref{introthm: EE disintegration}}

	First, we recall some measure theory background needed to create a measurable isomorphism that will exhibit fiberwise elementary equivalence. Recall that a \emph{standard Borel space} is a measurable space isomorphic to a complete separable metric space equipped with its Borel $\sigma$-algebra. Any two such spaces, if they have no atoms, are Borel-isomorphic, so that we can always assume without loss of generality that we are working on $[0,1]$ with its Borel $\sigma$-algebra.
	
	\begin{thm}[Rokhlin disintegration theorem] \label{thm: Rokhlin}
		Let $\mu$ be a Borel measure on $[0,1] \times [0,1]$.  Let $\mu_1$ be the pushforward of $\mu$ under the first coordinate projection.  Then there exists a family of probability measures $(\nu_x)_{x \in [0,1]}$ such that
		\begin{itemize}
			\item For each Borel $A \subseteq [0,1]$, $x \mapsto \nu_x(A)$ is Borel-measurable.
			\item For each Borel set $A$, let $A_x = \{y \in [0,1]: (x,y) \in A\}$.  Then $A_x$ is Borel measurable and
			\[
			\mu(A) = \int_{[0,1]} \nu_x(A_x)\,d\mu_1(x).
			\]
		\end{itemize}
	\end{thm}
	
	One can easily obtain the following equivalent notions of a probability space having ``diffuse fibers'' relative to a given random variable $X$.  This result could also be deduced from certain versions of the spectral theorem or from direct integral decompositions of von Neumann algebras, but we prefer to give a more self-contained proof.  Note that condition (3) is exactly the stability under tensorization assumed in Theorem \ref{introthm: EE disintegration}.
	
	\begin{lemma} \label{lem: diffuse fibers}
		Consider $[0,1]$ with Lebesgue measure.  Let $X$ be any Borel random variable $[0,1] \to [0,1]$ and let $Y$ be the identity function $[0,1] \to [0,1]$.  Let $\mu_X \in \mathcal{P}([0,1])$ be the distribution of $X$, that is, the pushforward of Lebesgue measure by the function $X$. Similarly, let $\mu_{X,Y} \in \mathcal{P}([0,1] \times [0,1])$ be the distribution of $(X,Y)$, or the pushforward Lebesgue measure under $(X,Y)$.  Let $\pi_1, \pi_2: [0,1] \times [0,1] \to [0,1]$ denote the coordinate projections.  The following are equivalent:
		\begin{enumerate}
			\item For $\mu_X$-almost every $x$, the measure $\nu_x$ in the disintegration for $\mu_{X,Y}$ from Theorem \ref{thm: Rokhlin} has no atoms.
			\item There exists an isomorphism of measure spaces $\Phi: ([0,1],\Leb) \to ([0,1],\mu_X) \times ([0,1],\Leb)$ such that $X = \pi_1 \circ \Phi$.
			\item There exists an isomorphism of measure spaces $\Psi: ([0,1],\Leb) \to ([0,1],\Leb) \times ([0,1],\Leb)$ such that $X = X \circ \pi_1 \circ \Psi$.
			\item There exists an isomorphism of measure spaces $\Psi: ([0,1],\Leb) \to ([0,1],\Leb) \times ([0,1],\Leb)$ and a random variable $Z$ such that $X = Z \circ \pi_1 \circ \Psi$.
			\item There exists a sequence of random variables $Z_n: [0,1] \to \{-1,1\}$ such that for all random variables $W: [0,1] \to [0,1]$, we have $\norm{E[Z_n W \mid X]}_{L^2} \to 0$.
		\end{enumerate}
	\end{lemma}
	
	\begin{proof}
		(1) $\implies$ (2).  Note that $([0,1],\Leb)$ is isomorphic to $([0,1] \times [0,1], \mu_{X,Y})$ as a measure space since $(X,Y)$ generate the entire Borel $\sigma$-algebra on $[0,1]$.  Hence, it suffices to produce an isomorphism $([0,1] \times [0,1], \mu_{X,Y}) \to ([0,1],\mu_X) \times ([0,1],\Leb)$ that preserves the first coordinate.
		
		Let $F_{\nu_x}: [0,1] \to [0,1]$ be the cumulative distribution function of $\nu_x$, that is, $F_{\nu_x}(y) = \nu_x([0,y])$.  Since $\nu_x$ has no atoms, $F_{\nu_x}$ is continuous.  Since $F_{\nu_x}$ is an increasing surjective continuous function $[0,1] \to [0,1]$, there is a unique strictly increasing, right-continuous function $G_x: [0,1] \to [0,1]$ such that $F_{\nu_x} \circ G_x = \id$.  Let $F(x,y) = F_{\nu_x}(y)$ and let $G(x,y) = G_x(y)$.  We claim that $F$ and $G$ are Borel-measurable.  Indeed, note that
		\[
		\{(x,y): F(x,y) \leq t\} = \{(x,y): \nu_x([0,y]) \leq t\},
		\]
		which is Borel-measurable by the properties of the disintegration.  On the other hand,
		\[
		\{(x,y): G(x,y) \leq t \} = \{(x,y): y \leq F(x,t) \},
		\]
		which is Borel-measurable since $F$ is Borel-measurable.  Note that $(F_{\nu_x})_* \nu_x = \Leb$ since $(F_{\nu_x})_* \nu_x([0,t]) = \nu_x(F_{\nu_x}^{-1}([0,t])) = \nu_x([0,G_x(t)]) = F_{\nu_x}(G_x(t)) = t$.
		
		Now let $\tilde{F}(x,y) = (x, F(x,y))$ and $\tilde{G}(x,y) = (x,G(x,y))$. Note that $F(x,G(x,y)) = y$ by construction, and $\mu_{X,Y}$-almost everywhere we have $G(x,F(x,y)) = y$ since the set of jump discontinuities of $G_x$ has measure zero, and hence $\tilde{F}$ and $\tilde{G}$ are inverses of each other almost everywhere.  Moreover, $\tilde{F}_* \mu_{X,Y} = \mu_X \times \Leb$ because using the disintegration theorem, for $h \in C([0,1] \times [0,1])$,
		\begin{align*}
			\int_{[0,1]^2} h \circ \tilde{F}\,d\mu_{X,Y} &= \int_{[0,1]} \int_{[0,1]} h(x,F(x,y))\,d\nu_x(y)\,d\mu_X(x) \\
			&= \int_{[0,1]} \int_{[0,1]} h(x,y) \,d(F_{\nu_x})_*\nu_x(y)\,d\mu_X(x) \\
			&= \int_{[0,1]} \int_{[0,1]} h(x,y)\,dy\,d\mu_X(x) \\
			&= \int_{[0,1] \times [0,1]} h\,d(\mu_X \times \Leb).
		\end{align*}

		(2) $\implies$ (3).  Let $\Phi$ be as in (2).  Let $\Gamma: [0,1] \to [0,1] \times [0,1]$ be any isomorphism of measure spaces (for instance, one can take $\Gamma = \Phi$).  Then let $\Psi = (\Phi^{-1} \times \id_{[0,1]}) \circ (\id_{[0,1]} \times \Gamma) \circ \Phi$.  We leave the verification to the reader.
		
		(3) $\implies$ (4).  This is immediate.
		
		(4) $\implies$ (5).  Let $\Psi$ be as in (4), and let $\tilde{X} = X \circ \pi_1 \circ \Psi$.  We can use the maps $\Psi$ and $\Psi^{-1}$ to transform between random variables on $[0,1] \times [0,1]$ and $[0,1]$, hence it suffices to show that there are random variables $Z_n: [0,1] \times [0,1] \to \{-1,1\}$, such that for all random variables $W: [0,1] \times [0,1] \to [0,1]$, we have $\norm{E[Z_n W \mid \tilde{X}]}_{L^2} \to 0$.  Let $f_n: [0,1] \to \{-1,1\}$ be a function that alternates between $\pm 1$ on intervals of length $2^{-n}$.  Let $Z_n = f_n \circ \pi_2$.  Given Borel $W_1$, $W_2: [0,1] \to \R$, consider the tensor product $W_1 \otimes W_2$ on $[0,1] \times [0,1]$.  Then
		\begin{multline*}
			\norm{E[Z_n (W_1 \otimes W_2) \mid \tilde{X}]}_{L^2} \leq \norm{E[Z_n(W_1 \otimes W_2) \mid \pi_1]}_{L^2} \\
			= \norm{(W_1 \otimes 1) E[Z_n W_2]}_{L^2} = \norm{W_1}_{L^2} | E[Z_n W_2]| \to 0.
		\end{multline*}
		Since simple tensors span a dense subset of the probability space, the claim follows.
		
		(5) $\implies$ (1).  Use the same notation from the first step. Let $S_\epsilon$ be the set of $(x,y)$ such that $\nu_x$ has an atom of size $\geq \epsilon$ at $y$.  We can write
		\[
		S_\epsilon = \{(x,y): \lim_{z \to y^-} F_{\nu_x}(z) \leq F_{\nu_x}(y) - \epsilon \},
		\]
		and hence $S_\epsilon$ is a Borel set since $(x,y) \mapsto F_{\nu_x}(y)$ is Borel-measurable.  Let $(S_\epsilon)_x$ be the slice $\{y: (x,y) \in S_\epsilon\}$.  From this it is not hard to see that
		\[
		f_\epsilon(x) = \begin{cases} \inf (S_\epsilon)_x, & (S_\epsilon)_x \neq \varnothing, \\ -1, \text{ else.} \end{cases}
		\]
		is a Borel function.  Let $(Z_n)_{n \in \N}$ be a sequence of random variables as in (5).  Let $W_\epsilon = \mathbf{1}_{Y = f_\epsilon(X)}$.  By passing to a subsequence, assume that $E[W_\epsilon Z_n \mid X] \to 0$ almost surely (that is, $\mu_X$-almost everywhere).  But note that when $(S_\epsilon)_x \neq \varnothing$, we have
		\[
		|E[W_\epsilon Z_n \mid X](x)| = \left| \int_{[0,1]} Z_n(x,y) \mathbf{1}_{y = f_\epsilon(x)}\,d\nu_x(y) \right| = |Z_n(x,f_\epsilon(x)) \nu_x(\{f_\epsilon(x)))| \geq \epsilon.
		\]
		Therefore, whenever $E[W_\epsilon Z_n \mid X](x) \to 0$, we have $(S_\epsilon)_x = \varnothing$.  Hence, for $\mu_X$-almost every $x$, we have $(S_\epsilon)_x = \varnothing$.  Since this holds for every $\epsilon$, we conclude that for $\mu_X$-almost every $x$, the measure $\nu_x$ has no atoms.
	\end{proof}

	\begin{proof}[Proof of Theorem \ref{introthm: EE disintegration}]
		The measure space $(\Omega_1,\mu_1)$ can be decomposed into countably many atoms and an atomless space, which we in turn can assume is $([0,1],\Leb)$.  Thus, suppose that $\Omega_1 = I \sqcup [0,1]$ and
		\[
		(M,\tau) = \bigoplus_{i \in I} \alpha_i (M_i,\tau_i) \oplus \alpha_0 \int_{[0,1]}^{\oplus} (M_\omega,\tau_\omega)\,d\omega,
		\]
		where $(\alpha_i)_{i \in I}$ are the weights of the atoms and $\alpha_0 = 1 - \sum_{i \in I} \alpha_i$.  Let $(M_0,\tau_0) = \int_{[0,1]}^{\oplus} (M_\omega,\tau_\omega)\,d\omega$. Similarly, suppose that
		\[
		(N,\sigma) = \bigoplus_{j \in J} \beta_i (M_j,\sigma_j) \oplus \beta_0 \int_{[0,1]}^{\oplus} (N_\omega,\sigma_\omega)\,d\omega,
		\]
		and let $(N_0,\sigma_0) = \int_{[0,1]}^{\oplus} (N_\omega,\sigma_\omega)\,d\omega$.
		
		Since we assumed $(M,\tau) \equiv (N,\sigma)$, the Keisler-Shelah theorem yields some ultrafilter $\cU$ on some index set $S$ such that $(M,\tau)^{\cU} \cong (N,\sigma)^{\cU}$.  Note that there is an isomorphism
		\[
		(M,\tau)^{\cU} \cong \bigoplus_{i \in \{0\} \cup I} \alpha_i (M_i,\tau_i)^{\cU}
		\]
		in the obvious way.  Indeed, any element in the ultrapower $(M,\tau)^{\cU}$ can be lifted to some tuple $(x_s)_{s \in S}$ over the index set $S$, which in turn can be decomposed as a direct sum of elements $x_{s,i} \in M_i$,  $i \in I \cup \{0\}$.  Moreover, $\norm{x_s}_2 \to 0$ along the ultrafilter if and only if $\norm{x_{s,i}}_2 \to 0$ along the ultrafilter for each $i$.  The elements $(M_i,\tau_i)^{\cU}$ are factors for each $i \in I$ (by Proposition \ref{prop: center of ultrafiber} for instance), while $(M_0,\tau_0)^{\cU}$ has diffuse center.  Thus, the elements $1_{M_i}$ are exactly the minimal central projections in $(M,\tau)^{\cU}$ for $i \in I$.
		
		The analogous claims hold for $(N,\sigma)$ as well.  Hence, $(M,\tau)^{\cU} \cong (N,\sigma)^{\cU}$ becomes
		\[
		\bigoplus_{i \in I} \alpha_i (M_i,\tau_i)^{\cU} \oplus \alpha_0 (M_0,\tau_0)^{\cU} \cong \bigoplus_{j \in J} \beta_j (N_j,\sigma_j)^{\cU} \oplus \beta_0 (N,\sigma)^{\cU}.
		\]
		The isomorphism must map minimal central projections on the left-hand side to minimal central projections on the right-hand side.  Therefore there is a bijection $f_1: I \to J$ such that the identity in $(M_i,\tau_i)^{\cU}$ is mapped to the identity in $(N_{f_1(i)},\sigma_{f_1(i)})^{\cU}$ under the isomorphism.  It easily follows that $\alpha_i = \beta_{f_1(i)}$ and $(M_i,\tau_i)^{\cU} \cong (N_{f_1(i)},\sigma_{f_1(i)})^{\cU}$, which means that $(M_i,\tau_i) \equiv (N_{f(i)},\sigma_{f(i)})$.  Moreover, the remaining summands with diffuse center in the direct sum decomposition must also correspond under the isomorphism, so that $(M_0,\tau_0)^{\cU} \cong (N_0,\sigma_0)^{\cU}$, which means that $(M_0,\tau_0) \equiv (N_0,\sigma_0)$.
		
		By Corollary \ref{col: same distribution of theories}, $(M_0,\tau_0)$ and $(N_0,\sigma_0)$ have the same distribution of theories.  By Lemma \ref{lem: direct integral distribution of theories}, this means that the random variables $X: \omega \mapsto \Th(M_\omega,\tau_\omega)$ and $Y: \omega \mapsto \Th(N_\omega,\sigma_\omega)$ have the same distribution, or $\mu_X = \mu_Y$ as probability measures on $\mathbb{T}_{\tr}$.  Let $g: \mathbb{T}_{\tr} \to [0,1]$ be an isomorphism of standard Borel spaces, and let $\tilde{X} = g \circ X$.  Since we assumed that the diffuse part of the direct integral decomposition satisfies stability under tensorization, it is easy to see that the random variable $\tilde{X}$ satisfies condition (3) of Lemma \ref{lem: diffuse fibers}.  Hence, by condition (2), there exists a measurable isomorphism $\tilde{\Phi}: ([0,1],\Leb) \to ([0,1],\mu_{\tilde{X}}) \times ([0,1],\Leb)$ such that $\tilde{X} = \pi_1 \circ \tilde{\Phi}$.  Let $\Phi = (g^{-1} \times \id) \circ \tilde{\Phi}: ([0,1],\Leb) \to (\mathbb{T}_{\tr},\mu_{X}) \times ([0,1],\Leb)$; then we have $X = \pi_1 \circ \Phi$.  Similarly, obtain an isomorphism $\Psi: ([0,1],\Leb) \to (\mathbb{T}_{\tr},\mu_{Y}) \times ([0,1],\Leb)$ such that $Y = \pi_1 \circ \Psi$.
		
		Then $f_0 = \Psi^{-1} \circ \Phi$ gives an isomorphism $([0,1],\Leb) \to ([0,1],\Leb)$ such that $Y \circ f_0 = X$, meaning that $\Th(N_{f_0(\omega)},\sigma_{f_0(\omega)}) = \Th(M_\omega,\tau_\omega)$.  Let $f: I \sqcup [0,1] \to J \sqcup [0,1]$ be the function given by $f|_I = f_1$ and $f|_{[0,1]} = f_0$.  Then $f$ is a bijection such that $(N_{f(\omega)},\sigma_{f(\omega)}) \equiv (M_\omega, \tau_\omega)$. 
	\end{proof}
	
	\begin{rmk} \label{rmk: Thm A general case}
		In the setup of Theorem \ref{introthm: EE disintegration}, if we remove the hypothesis of stability under tensorization for the diffuse part of the direct integral decomposition, can we still find a measurable isomorphism $f: \Omega_1 \to \Omega_2$ such that $M_\omega \equiv N_{f(\omega)}$ almost surely?  We claim that the answer comes down to the behavior of the fiber measures in the distintegration of the diffuse parts of $\Omega_1$ and $\Omega_2$ over the random variables $X$ and $Y$ respectively.
		
		The proof given above still allows us to match up the direct summands over each atom in the probability spaces $(\Omega_1,\mu_1)$ and $(\Omega_2,\mu_2)$, and so it remains to consider what happens for the diffuse parts.  So for simplicity, assume that there are no atoms in $\Omega_1$ and $\Omega_2$.  Assume $\Omega_1$ and $\Omega_2$ are standard Borel spaces.  Disintegrate $\mu_1$ with respect to the random variable $X(\omega) = \Th(M_\omega,\tau_\omega)$ as $\mu_1 = \int_{\mathbb{T}_{\tr}} \nu_x \,d\mu_X(x)$, where $\nu_x \in \mathcal{P}(\Omega_1)$ is supported on $S_x = \{\omega \in \Omega_1: X(\omega) = x\}$.  This is possible since $\Omega_1$ and $\mathbb{T}_1$ are standard Borel spaces, hence isomorphic to $[0,1]$ as measure spaces, and for the same reason, Lemma \ref{lem: diffuse fibers} also applies to this disintegration.  In particular, the hypothesis of stability under tensorization means precisely that for $\mu_X$-a.e.\ $x$, the fiber measure $\nu_x$ is diffuse.  Similarly, let $\mu_2 = \int_{\mathbb{T}_{\tr}} \tilde{\nu}_y\,d\mu_Y(y)$ where $\tilde{\nu}_y$ is supported on $\tilde{S}_y = \{\omega \in \Omega_2: Y(\omega) = y\}$.  As in the proof above, we still have $\mu_X = \mu_Y$.  By construction, for each $x \in \mathbb{T}_{\tr}$, the tracial von Neumann algebras $M_\omega$ for $\omega S_x$ and $N_\omega$ for $\omega \in S_y$ have theory equal to $x$, and so they are all elementarily equivalent to each other.
		
		Thus, if the sizes of the atoms of the measures $\nu_x$ and $\tilde{\nu}_x$ agree for almost every $x$, then there is an isomorphism of $(S_x,\nu_x)$ to $(\tilde{S}_x,\tilde{\nu}_x)$.  In this case, one can show that these maps can be chosen to depend measurably on $x$, or there is a measurable isomorphism $f: (\Omega_1,\mu_1) \to (\Omega_2,\mu_2)$ sending $S_x$ to $\tilde{S}_x$ (we leave the details to the reader, but the idea is to measurably pick out atoms of a certain size in each fiber as in the proof of Lemma \ref{lem: diffuse fibers} (5) $\implies$ (1)).  Then $X(\omega) = \Th(M_{\omega}) =Y(f(\omega)) = \Th(N_{f(\omega)})$, so that $M_{\omega} \equiv N_{f(\omega)}$ almost everywhere.  Conversely, if such a measurable isomorphism exists, then it is necessary that for $\mu_X$-a.e.\ $x \in \mathbb{T}_{\tr}$, the fiber measures $\nu_x$ and $\tilde{\nu}_x$ have the same number of atoms of each size.
	\end{rmk}
	
	\section{Isomorphism classes of ultrafibers} \label{sec: saturation}
	
	In this section, we show that under the continuum hypothesis (CH), the ultrafibers obtained from separable algebras are all isomorphic to some ultraproduct over the natural numbers.  More precisely, we show the following.
	
	\begin{prop} \label{prop: isomorphism}
		Assume CH.  Let $(\Omega,\mu)$ be a standard Borel probability space and let $\cU$ be a pure state on $L^\infty(\Omega,\mu)$.  Let $(M_\omega)_{\omega \in \Omega}$ be a measurable field of separable tracial von Neumann algebras and write $(M,\tau) = \int_\Omega^{\oplus} (M_\omega,\tau_\omega)\,d\mu(\omega)$, and let $E: M \to L^\infty(\Omega,\mu)$ be the canonical expectation. Then there exists a sequence $(\omega_n)_{n \in \N}$ in $\Omega$ and a non-principal ultrafilter $\cV$ on $\N$ such that $M^{/E,\cU} \cong \prod_{n \to \cV} M_{\omega_n}$.
	\end{prop}
	
	We remark that Proposition \ref{prop: isomorphism} in particular applies to generalized ultrapowers over separable commutative tracial von Neumann algebras as follows.  Suppose that $(M,\tau_M)$ is a tracial von Neumann algebra.  Let $A$ be a separable diffuse commutative tracial von Neumann algebra, and let $\cU$ is a character on $A$.  We can assume without loss of generality that $A = L^\infty(\Omega)$, and then $A \overline{\otimes} M$ is identified with $\int_\Omega^\omega (M,\tau_M) \,d\omega$, and the ultrafiber $M^{/E,\cU}$ is exactly the generalized ultrapower $M^{\cU}$ of Definition \ref{def: gen ultrapower}.  Hence, the proposition implies that $M^{\cU}$ is isomorphic to an ultrapower $M^{\mathcal{V}}$ with respect to a non-principal ultrafilter $\mathcal{V}$ on $\N$.
	
	The isomorphism in Theorem \ref{prop: isomorphism} arises from the following classic fact from model theory.  The notions of density character and saturation are explained below.
	
	\begin{prop}[{\cite[Proposition 4.13]{FHS2014a}}] \label{saturation-iso}
		If $(M, \tau_M)$ and $(N, \tau_N)$ are saturated tracial von Neumann algebras of the same uncountable density character, then $M \equiv N$ if and only if $M \cong N$.
	\end{prop}
	
	Recall that the \emph{density character} of a metric space is the minimum cardinality of a dense subset, or equivalently the maximum cardinality of a collection of disjoint balls.  The density character of a tracial von Neumann algebra is understood with respect to the $L^2$-metric, or equivalently, it refers to the density character of the predual.  Thus, the first prerequisite for applying Proposition \ref{saturation-iso} is to determine the density character of the ultrafibers in question.  We will show the density character is $2^{\aleph_0}$ if the algebra has a diffuse part, and in fact, this is true even without assuming CH.
	
	\begin{prop} \label{prop: density character}
		Let $(\Omega,\mu)$ be a standard Borel probability space and let $\cU$ be a character on $L^\infty(\Omega,\mu)$.  Assume that $\cU$ is not given by a point mass located at some atom of $\Omega$.  Let $(M_\omega)_{\omega \in \Omega}$ be a measurable field of separable tracial von Neumann algebras and write $(M,\tau) = \int_\Omega^{\oplus} (M_\omega,\tau_\omega)\,d\mu(\omega)$, and let $E: M \to L^\infty(\Omega,\mu)$ be the canonical expectation.
		\begin{enumerate}
			\item If $M^{/E,\cU}$ has a diffuse summand, then its density character is $2^{\aleph_0}$.
			\item Otherwise, $M^{/E,\cU}$ is a direct sum of matrix algebras, and hence is separable.
		\end{enumerate}
	\end{prop}
	
	For the proof, we need one more property of the ultrafilters or characters on $L^\infty(\Omega)$, namely countable incompleteness.  Recall that an ultrafilter $\mathcal{U}$ on an index set $I$ is called \textit{countably incomplete} if there exists a decreasing sequence of sets $\{A_n\}_{n = 1}^\infty \subseteq \mathcal{U}$ such that $\bigcap_n A_n = \varnothing$. For instance, all non-principal ultrafilters on $\mathbb{N}$ are countably incomplete since we can take $A_n = \{n,n+1,\dots\}$.  If we view the ultrafilter as a pure state on $\ell^\infty(I)$, then countable incompleteness means that there is a decreasing sequence of projections $p_n$ such that $\cU(p_n) = 1$ and $\bigwedge_{n \in \N} p_n = 0$.  Thus, we adapt the definition to characters on commutative von Neumann algebras as follows.
	
	\begin{defn}
		Let $\mathcal{U}$ be a character on a commutative von Neumann algebra $N$. We say $\mathcal{U}$ is \textit{countably incomplete} if there exists a sequence of decreasing projections $\{p_n\}_{n = 1}^\infty \subseteq N$ s.t. $p_n \searrow 0$ and $\mathcal{U}(p_n) = 1$ for all $n$.
	\end{defn}
	
	\begin{prop}
		Let $N$ be a commutative tracial von Neumann algebra and $\mathcal{U}$ a character on $N$. Then $\mathcal{U}$ is countably incomplete unless there is a minimal projection $p$ such that $px = \cU(x) p$ for all $x \in N$.
	\end{prop}
	
	\begin{proof}
		First, consider the case where $N$ is diffuse.  By Maharam's theorem \cite[Theorem 1]{Mah42}, we may assume $(N, \tau_N) = L^\infty([0, 1]^\kappa, \lambda^{\otimes \kappa})$ for some infinite cardinal $\kappa$, where $\lambda$ is the Lebesgue measure. Thus, $A = C([0, 1]^\kappa) \subseteq N$, so $\mathcal{U}$ necessarily restricts to a character on $A$, which must be of the form of the evaluation functional at some $t \in [0, 1]^\kappa$. For each positive integer $n$, let $f_n: [0, 1] \to [0, 1]$ be a continuous function s.t. $f_n = 1$ on $[t(0) - 1/(2n), t(0) + 1/(2n)] \cap [0, 1]$ and $f_n = 0$ outside $[t(0) - 1/n, t(0) + 1/n] \cap [0, 1]$. Let $p: [0, 1]^\kappa \to [0, 1]$ be the projection map onto the zeroth coordinate. Then $g_n = f_n \circ p \in A$. Furthermore, by definition $\mathcal{U}(g_n) = g_n(t) = 1$. Let $p_n \in N$ be the projection corresponding to the set $([t(0) - 1/n, t(0) + 1/n] \cap [0, 1]) \times [0, 1]^{\kappa \setminus \{0\}}$. Then $p_n \geq g_n$, so $\mathcal{U}(p_n) = 1$. We also have $\tau_N(p_n) \leq 2/n \to 0$, so $p_n \searrow 0$.
		
		Now consider the general case.  Choose a family $(p_i)_{i \in I}$ of mutually orthogonal minimal projections in $N$ that is maximal with respect to inclusion.  Since $N$ has a faithful state, $I$ must be countable.  Let $p = \sum_{i \in I} p_i$.  Then $N$ is the direct sum of $pN$ and $(1-p)N$, and so the character $\cU$ must be supported on one of these two summands.  If it is supported on $(1-p)N$, then we can apply the diffuse case.  If it is supported on $pN$, then the argument is the same as for $\ell^\infty(\N)$.
	\end{proof}
	
	\begin{proof}[Proof of Proposition \ref{prop: density character}]
		(1) Suppose that $p$ is a central projection such that $p M^{/E,\cU}$ is diffuse, and hence it contains a copy of $L^\infty(\{-1,1\}^{\N})$ where $\{-1,1\}^{\N}$ is the probability space equipped with the infinite tensor product of the uniform measure $\frac{1}{2} (\delta_{-1} + \delta_1)$.  Let $v_k: \{-1,1\}^{\N} \to \{-1,1\}$ be the $k$th coordinate projection, viewed as an element of $pM^{/E,\cU}$.  Note that $v_k$ is self-adjoint with $v_k^2 = p$, and the $v_k$'s commute.  Since $M^{/E,\cU}$ is a quotient of $M$, write
		\[
		v_k = \pi_{E,\cU} \left[ \int_{\Omega}^{\oplus} v_k(\omega)\,d\mu(\omega) \right],
		\]
		where $v_k(\omega) \in (M_\omega)_1$ is self-adjoint.  For $n \in \N$, let $A_n \subseteq \Omega$ be the set of $\omega$ such that
		\[
		\left| \tau_\omega\left[ \prod_{j=1}^m v_{i(j)}(\omega) \right] - \tau_{E,\cU}\left[ \prod_{j=1}^m v_{i(j)} \right] \right| \leq \frac{1}{n} \text{ for } m \leq 2n \text{ and } i(1), \dots, i(m) \in \{1,\dots,n\}.
		\]
		Note that because $\cU$ evaluated on $\omega \mapsto \tau_\omega\left[ \prod_{j=1}^m v_{i(j)}(\omega) \right]$ is equal to $\tau_{E,\cU}\left[ \prod_{j=1}^m v_{i(j)} \right]$, we must have that
		\[
		\cU(\mathbbm{1}_{A_n}) = 1.
		\]
		Also, $A_n$ is a decreasing sequence.  Moreover, since $\cU$ is countably incomplete, choose a decreasing collection of subsets $B_n$ such that $\cU(\mathbbm{1}_{B_n}) = 1$ and $\mu(B_n) \to 0$.  Let $C_n = A_n \cap B_n$, so that $\cU(\mathbbm{1}_{C_n}) = 1$ and $\mu(C_n) \to 0$. 
		For each $S \subseteq \N$, define
		\[
		v_S(\omega) = \sum_{n \in \N} \mathbbm{1}_{C_n \setminus C_{n+1}}(\omega) \prod_{i \in S \cap \{1,\dots,n\}} v_i(\omega),
		\]
		where the product runs from left to right when the indices are written in increasing order.  Suppose the sets $S$ and $S'$ are not equal, and say they differ at some integer $m$.   Fix some $n_0 \geq m$.  Then for $n \geq n_0$, we have
		\[
		\tau_{E,\cU}\left[ \left( \prod_{i \in S \cap \{1,\dots,n\}} v_i \right)^* \prod_{i \in S' \cap \{1,\dots,n\}} v_i \right] = 0,
		\]
		and hence by definition of $A_n$, we have for $\omega \in C_n \setminus C_{n+1}$ that
		\[
		|\tau_\omega(v_{S}(\omega)^* v_{S'}(\omega))| = \left|\tau_\omega\left[ \left( \prod_{i \in S \cap \{1,\dots,n\}} v_i(\omega)\right)^* \prod_{i \in S' \cap \{1,\dots,n\}} v_i(\omega) \right] \right| \leq \frac{1}{n} \leq \frac{1}{n_0}.
		\]
		This holds on $\bigcup_{n \geq n_0} (C_n \setminus C_{n+1})$, which equals $C_{n_0}$ up to null sets.  Since $\cU(\mathbbm{1}_{C_{n_0}}) = 1$, we have that
		\[
		|\tau_{E,\cU}(v_S^* v_{S'})| \leq \frac{1}{n_0}.
		\]
		Since $n_0$ was arbitrary, $\tau_{E,\cU}(v_S^* v_{S'}) = 0$.  Similar reasoning shows that $\tau_{E,\cU}(v_S^*v_S) = \tau_{E,\cU}(p)$.  Therefore, $(v_S)_{S \subseteq \N}$ gives an orthogonal set of nonzero elements in $pM$ of cardinality equal to $2^{\aleph_0}$, so the density character is at least $2^{\aleph_0}$.
		
		On the other hand, since $M$ is separable, its cardinality is at most $2^{\aleph_0}$.  And $M^{/E,\cU}$ is the image of $M$ under the quotient map, and so its cardinality is at most $2^{\aleph_0}$.  Hence, the density character is also at most $2^{\aleph_0}$.
		
		(2) Suppose $M^{/E,\cU}$ has no diffuse summand.  Choose a family $(p_i)_{i \in I}$ of mutually orthogonal minimal projections in $Z(M^{/E,\cU})$ that is maximal with respect to inclusion.  This collection must be countable since $M^{/E,\cU}$ has a faithful trace.  We must have $\sum_{i \in I} p_i = 1$ or else we could construct a minimal projection under the complement and contradict maximality.  Note that $p_i M^{/E,\cU}$ must be a factor or else it would contain a nontrivial central projection and contradict minimality of $p_i$.  Thus, $p_i M^{E,\cU}$ must be a tracial factor that admits a minimal projection, and so is isomorphic to a matrix algebra.
	\end{proof}
	
	Now that we have computed the density character of the generalized ultraproducts, we move on to saturation.  Roughly speaking, saturation means that if some conditions have approximate solutions, then they have a solution; saturation is thus somewhat similar to compactness (see \cite[Proposition 2.28]{JekelDefinable}).  The number of constants that can appear in the conditions is a specified cardinality $\kappa$.  The definition is as follows; see \cite[Definition 7.5]{BYBHU08}, \cite[\S 4.4]{FHS2014a}.
	
	\begin{defn}
		Let $M$ be a tracial von Neumann algebra.  Let $A \subseteq (M)_1$ and let $\Gamma$ be a collection of formulas whose free variables lie within $\{x_1, \cdots, x_n\} \cup \{y_a\}_{a \in A}$.  with free variables $x_1,\dots,x_n$ and $(y_i)_{i \in I}$.
		\begin{enumerate}
			\item The collection $\Gamma$ is \emph{satisfiable} in $M$ if there exist $x_1$, \dots, $x_n \in (M)_1$, such that
			\[
			\varphi^M(x_1, \cdots, x_n, \{a\}_{a \in A}) = 0 \text{ for } \varphi \in \Gamma.
			\]
			\item The collection $\Gamma$ is \emph{finitely approximately satisfiable} if, for every finite subset $\Gamma_0 \subseteq \Gamma$ and every $\epsilon > 0$, there exists $\{x_1, \cdots, x_n\} \subseteq (M)_1$, such that
			\[
			|\varphi^M(x_1, \cdots, x_n, \{a\}_{a \in A})| < \epsilon \text{ for } \varphi \in \Gamma_0.
			\]
			\item Let $\kappa$ be an infinite cardinal.  We say that $M$ is \emph{$\kappa$-saturated} if for every $A \subseteq (M)_1$ with $|A| < \kappa$ and every collection $\Gamma$ as above, if $\Gamma$ is finitely approximately satisfiable, then $\Gamma$ is satisfiable.
			\item We call $M$ \emph{countably saturated} if it is $\aleph_1$-saturated, meaning that the above condition holds whenever $|A| < \aleph_1$, i.e., when $A$ is countable.
			\item We call $M$ \emph{saturated} if it is $\kappa$-saturated where $\kappa$ is the density character of $M$.
		\end{enumerate}
	\end{defn}
	
	Hence, in Proposition \ref{saturation-iso}, the requirement is for $M$ and $N$ to be saturated with the cardinal specified by their density character.  We showed in Proposition \ref{prop: density character} that the density character of ultrafibers arising from a direct integral is $2^{\aleph_0}$ in the non-atomic case.  Under CH, this equals $\aleph_1$, so in this case \emph{countable} saturation is enough to obtain saturation.  It is well known that ultraproducts with respect to countably incomplete ultrafilters are countably saturated \cite[Proposition 7.6]{BYBHU08}, \cite[Proposition 4.11]{FHS2014a}.  Our next goal is to show the same conclusion for ultrafibers over continuous index sets by adapting the classic diagonalization argument.
	
	\begin{prop}\label{saturation-thm}
		Let $M$ be a finite von Neumann algebra, $N \subseteq Z(M)$ be a subalgebra of its center, $E: M \to N$ be a normal, tracial, faithful conditional expectation, and $\mathcal{U}$ be a countably incomplete character on $N$. Then $M^{/E, \mathcal{U}}$ is countably saturated.
	\end{prop}
	
	\begin{proof}
		Let $A \subset (M^{/E, \mathcal{U}})_1$ be a countable set. For each $a \in A$, fix a lift $\mathbbm{a} \in (M)_1$ of $a$. Let $\Gamma$ be a collection of formulae whose free variables lie within $\{x_1, \cdots, x_n\} \cup \{y_a\}_{a \in A}$, and suppose that $\Gamma$ is finitely approximately satisfiable in $M^{/E,\cU}$.  First, we consider the case where $\Gamma$ is countable, so write $\Gamma = \{\varphi_k: k \in \N\}$.
		
		For each $m \in \N$, since $\Gamma$ is finitely satisfiable, fix $x_1^{(m)},\dots,x_n^{(m)} \in (M^{/E, \mathcal{U}})_1$ such that
		\[
		|\varphi_k(x_1^{(m)},\dots,x_n^{(m)},(a)_{a \in A})| < \frac{1}{m} \text{ for } k = 1, \dots, m.
		\]
		Let $\mathbbm{x}_j^{(m)} \in (M)_1$ such that $\pi_{E,\cU}(\mathbbm{x}_j^{(m)}) = x_j^{(m)}$.  By Theorem \ref{thm: continuous Los}, we have
		\[
		|\mathcal{U}((\varphi_k)_E^M(\mathbbm{x}_1^{(m)}, \cdots, \mathbbm{x}_n^{(m)}, \{\mathbbm{a}\}_{a \in A}))| < \frac{1}{m} \text{ for } k = 1, \dots, m.
		\]
		By Lemma \ref{neighborhood-of-character}, there exists a projection $p_m \in N$ s.t. $\mathcal{U}(p_m) = 1$ and
		\[
		\|p_m(\varphi_k)_E^M(\mathbbm{x}_1^{(m)}, \cdots, \mathbbm{x}_n^{(m)}, \{\mathbbm{a}\}_{a \in A})\|_\infty < \frac{1}{m} \text{ for } k = 1, \dots, m.
		\]
		As $\mathcal{U}$ is countably incomplete, fix a sequence of projections $q_n \searrow 0$ with $\mathcal{U}(q_n) = 1$.  Then let $e_m = (\bigwedge_{k = 1}^m p_k) \wedge q_m$, so that $e_m \searrow 0$ and $e_m$ satisfies the same condition as $p_m$ above.
		
		Let $e_0' = 1 - e_1$ and $e_m' = e_m - e_{m+1}$ for $m \geq 1$.  Then $(e_m')_{m=0}^\infty$ is a PVM over $N$.  For $i = 1, \dots, n$, define
		\[
		\mathbbm{x}_i = \sum_{m=0}^\infty e_m' \mathbbm{x}_i^{(m)},  \qquad x_i = \pi_{E,\cU}(\mathbbm{x}_i).
		\]
		Given $m_0 \in \N$, we have that $m \geq m_0$,
		\[
		\|e_m(\varphi_k)_E^M(\mathbbm{x}_1^{(m)}, \cdots, \mathbbm{x}_n^{(m)}, \{\mathbbm{a}\}_{a \in A})\|_\infty < \frac{1}{m} \text{ for } k = 1, \dots, m,
		\]
		hence
		\[
		\|e_m'(\varphi_k)_E^M(\mathbbm{x}_1^{(m)}, \cdots, \mathbbm{x}_n^{(m)}, \{\mathbbm{a}\}_{a \in A})\|_\infty < \frac{1}{m_0} \text{ for } k = 1, \dots, m_0.
		\]
		Therefore,
		\[
		\norm*{ e_{m_0} \sum_{m=0}^\infty e_m'(\varphi_k)_E^M(\mathbbm{x}_1^{(m)}, \cdots, \mathbbm{x}_n^{(m)}, \{\mathbbm{a}\}_{a \in A}) }_\infty \leq \frac{1}{m_0} \text{ for } k = 1, \dots, m_0.
		\]
		By Lemma \ref{formula-partition}, this means that
		\[
		\norm*{ e_{m_0} (\varphi_k)_E^M(\mathbbm{x}_1, \cdots, \mathbbm{x}_n, \{\mathbbm{a}\}_{a \in A}) }_\infty \leq \frac{1}{m_0} \text{ for } k = 1, \dots, m_0.
		\]
		Since $\cU(e_{m_0}) = 1$, we therefore have by Theorem \ref{thm: continuous Los} that
		\[
		|\varphi_k^{M^{/E,\cU}}(x_1,\dots,x_m,(a)_{a \in A})| = \left| \cU \left[ (\varphi_k)_E^M(\mathbbm{x}_1, \cdots, \mathbbm{x}_n, \{\mathbbm{a}\}_{a \in A}) \right] \right| \leq \frac{1}{m_0} \text{ for } k = 1, \dots, m_0.
		\]
		Since $m_0$ was arbitrary, $\varphi_k^{M^{/E,\cU}}(x_1,\dots,x_m,(a)_{a \in A}) = 0$ for all $k \in \N$, and therefore $\Gamma$ is satisfiable.
		
		It remains to remove the assumption that $\Gamma$ is countable.  Consider a general $\Gamma = (\varphi_i)_{i \in I}$ associated to the same countable set $A$.  By Lemma \ref{lem: separability}, choose a countable set of formulas $\Lambda$ that is dense with respect to $\norm{\cdot}_u$.  For each $i \in I$ and $m \in \N$, choose $\lambda_{i,m} \in \Lambda$ such that $\norm{\lambda_{i,m} - \varphi_i}_u \leq 1/m$.  Let
		\[
		\psi_{i,m} = \max(0, |\lambda_{i,m}| - 2/m).
		\]
		Since $\lambda_{i,m}$ comes from the countable set $\Lambda$, the set $\Gamma' = \{\psi_{i,m}: i \in I, m \in \N\}$ is countable.  Furthermore, note that for each $i \in I$,
		\[
		|\varphi_i| \leq \frac{1}{m} \implies \psi_{i,m} = 0 \implies |\varphi_i| \leq \frac{2}{m}.
		\]
		Thus, since $\Gamma$ is finitely approximately satisfiable, so is $\Gamma'$.  By the preceding argument $\Gamma'$ is satisfiable in $M$.  If $(x_1,\dots,x_n)$ satisfies $\Gamma'$, then it satisfies $|\varphi_i(x_1,\dots,x_m,(a)_{a \in A})| \leq 2/m$ for all $i \in I$ and $m \in \N$, and therefore, it satisfies $\Gamma$ as desired.
	\end{proof}
	
	Finally, we put the pieces together to complete the proof.
	
	\begin{proof}[Proof of Proposition \ref{prop: isomorphism}]
		Consider a direct integral $(M,\tau) = \int_\Omega^{\oplus} (M_\omega,\tau_\omega)\,d\mu(\omega)$.  Let $E: M \to L^\infty(\Omega,\mu)$ be the canonical conditional expectation.  First, we claim that
		\[
		\Th(M^{/E,\cU}) \in \overline{\{\Th(M_\omega): \omega \in \Omega\}}.
		\]
		Let $\ev_E^M: \mathcal{P}_{\varnothing} \to L^\infty(\Omega,\mu)$ be the evaluation map.  By Theorem \ref{thm: continuous Los}, $\Th(M^{/E,\cU})$ is the mapping $\mathcal{P}_{\varnothing} \to \R$ given by $\cU \circ \ev_E^M$.  For each $\varphi \in \mathcal{P}_{\varnothing}$, we have
		\[
		|\Th(M^{/E,\cU})[\varphi]| \leq \norm{\varphi_E^M}_{L^\infty(\Omega,\mu)} \leq \sup_{\omega \in \Omega} |\varphi^{M_\omega}| = \sup_{\omega \in \Omega} |\Th(M_\omega)[\varphi]|.
		\]
		Recall that $\varphi$ is equivalent to a continuous function on the compact Hausdorff space $\mathcal{X} = \Spec(\mathcal{P}_{\varnothing})$.  If $\mathcal{Y} \subseteq \mathcal{X}$ and we have $|f(x)| \leq \sup_{y \in \mathcal{Y}} |f(y)|$ for all $f \in C(\mathcal{X};\R)$, then $x \in \overline{\mathcal{Y}}$.  Since the theories are characters corresponding to point evaluations, this shows in our case that $\Th(M^{/E,\cU})$ is in the closure of $\{\Th(M_\omega): \omega \in \Omega\}$ as desired.
		
		Now since the space of theories is separable and metrizable, there exists a sequence $\omega_n$ such that $\Th(M_{\omega_n}) \to \Th(M^{/E,\cU})$.  Fix a non-principal ultrafilter $\cV$ on $\N$, and let $N = \prod_{\cV} (M_{\omega_n},\tau_{\omega_n})$.  By {\L}o{\'s}'s theorem, $\Th(N) = \lim_{n \to \cV} \Th(M_{\omega_n}) = \Th(M^{/E,\cU})$.
		
		By \cite[Lemma 4.4, proof of Proposition 4.3]{FG24}, we have that $N$ is atomic if and only if $M^{/E,\cU}$ is atomic; moreover, the sizes of the atoms in the decomposition and the dimensions of the matrix algebras are uniquely determined by the theory.  Thus, in the atomic case, we have $M^{/E,\cU} \cong N$.
		
		Otherwise, $M^{/E,\cU}$ and $N$ both have a diffuse summand.  By Proposition \ref{prop: density character}, the density character of $M^{/E,\cU}$ and of $N$ is the continuum $2^{\aleph_0}$.  Note that Proposition \ref{saturation-thm} implies that $M^{/E,\cU}$ and $N$ are countably saturated, i.e., $\aleph_1$-saturated.  By CH, $\aleph_1 = 2^{\aleph_0}$, and therefore $M^{/E,\cU}$ and $N$ are saturated, so by Proposition \ref{saturation-iso}, $M^{/E,\cU} \cong N$.
	\end{proof}
	
	\section{Random matrices and ultraproducts} \label{sec: RM}
	
	Our goal in this section is to describe how ultrafibers shed light on the relationship between random matrices and ultraproducts.  We denote by $\mathbb{M}_n$ the algebra of $n \times n$ complex matrices equipped with the normalized trace $\tr_n$.  Suppose that $X^{(n)} = (X_i^{(n)})_{i \in I}$ is a tuple of random $n \times n$ matrices on the probability space $([0,1],\Leb)$ and suppose that $\limsup_{n \to \infty} \norm{X_i^{(n)}} < \infty$ almost surely.  Fix also a free ultrafilter $\cU$ on $\mathbb{N}$.  Then in what sense does the sequence $X_i^{(n)}$ define a random element in the ultraproduct $\prod_{n \to \cU} \mathbb{M}_n$?
	
	For concreteness, assume that $X_i^{(n)}$ is an independent family of matrices from the Gaussian unitary ensemble (GUE) (for background, see e.g.\ \cite{AGZ2009,MingoSpeicher}).  For almost every $\omega \in [0,1]$, the sequence $(X_i^{(n)}(\omega))_{n \in \N}$ defines an element $X_i(\omega)$ in $\mathcal{Q} = \prod_{n \to \cU} \M_n$.  However, $\omega \mapsto X_i(\omega)$ is \emph{not} measurable in the usual sense.  Indeed, for $X_i(\omega)$ to be Bochner-measurable, it would have to almost surely take values in a separable subspace of $\mathcal{Q}$, which would imply there is some $L^2$ ball $B_r(Y)$ in $\mathcal{Q}$ such that $\{\omega: X_i(\omega) \in B_r(Y)\}$ has strictly positive measure with $r < 1$.  Write $Y = [Y^{(n)}]_{n \in \N}$.  Then this set can also be written as
	\begin{align*}
		\{\omega: \lim_{n \to \cU} \norm{X_i^{(n)} - Y^{(n)}}_2 < r\} &\subseteq \{\omega: \limsup_{n \to \cU} \norm{X_i^{(n)}(\omega) - Y^{(n)}}_2 < r\} \\
		&\subseteq \limsup_{n \to \infty} \{\omega: \norm{X_i^{(n)}(\omega) - Y^{(n)}}_2 < r\}
	\end{align*}
	Thus,
	\[
	\Leb(\omega: X_i(\omega) \in B_r(Y)) \leq \limsup_{n \to \infty} \Leb(\omega: \norm{X_i^{(n)}(\omega) - Y^{(n)}}_2 < r),
	\]
	which is zero because the Gaussian measure of the $r$-ball vanishes as $n \to \infty$ if $r < 1$.  It is difficult even to obtain \emph{weak}-$*$ measurability for $\omega \mapsto X_i(\omega)$ because the inner products with elements of $\mathcal{Q}$ are not ordinary limits of inner products with elements of $\M_n$, but rather \emph{ultralimits}.  Even if we could achieve weak-$*$ measurability, it would not be that helpful since the predual of the matrix ultraproduct is not separable; thus, many natural conditions we would want for the random matrix would end up producing uncountable unions and intersections.  Hence, great care has to be taken in arguments that combine probability and ultraproducts as in \cite{FJP2023,JP2024}.
	
	Therefore, rather than attempting to study $X_i$ as a map $[0,1] \to \prod_{n \to \cU} \M_n$, it seems much more natural to consider $X_i$ as an element of $\prod_{n \to \cU} (L^\infty[0,1] \otimes \M_n)$, which heuristically is more like a random variable on the non-separable probability space $([0,1],\Leb)^{\cU}$ corresponding to $L^\infty[0,1]^{\cU}$ than on the probability space $[0,1]$ itself.  However, it is not even a random variable over $([0,1],\Leb)^{\cU}$ in the ordinary sense.  And so, rather than trying to extract a deterministic element of $\prod_{n \to \cU} \M_n$ from $X_i$ by choosing a \emph{point} in some probability space, we can form an \emph{ultrafiber} of $\prod_{n \to \cU} (L^\infty[0,1] \otimes \M_n)$ using a pure state $\cV$ on its center $L^\infty[0,1]^{\cU}$.  As we will see below, the ultrafiber $M$ obtained from this process is \emph{larger} than the subalgebra $Q = \prod_{n \to \cU} \M_n \subseteq \prod_{n \to \cU} (L^\infty[0,1] \otimes \M_n)$ corresponding to the deterministic matrices.  In fact, the image of $X_i$ in $M$ will be freely independent of $Q$ (for background on free independence, see \cite{VDN1992,AGZ2009,MingoSpeicher}).
	
	\begin{prop}
		Let $\cU$ be a free ultrafilter on $\N$.  Let $N = \prod_{n \to \cU} (L^\infty[0,1] \otimes \M_n)$ and let $Q = \prod_{n \to \cU} \M_n \subseteq N$.  Let $\cV$ be a pure state on $Z(N) \cong L^\infty[0,1]^{\cU}$, let $E: N \to Z(N)$ be the trace-preserving conditional expectation, let $M = N^{/E,\cV}$ be the corresponding ultrafiber, and let $\pi_{E,\cV}: N \to M$ be the natural quotient map.
		
		Let $I$ be a countable index set and let $(X_i^{(n)})_{i \in I}$ be a family of independent GUE matrices on the probability space $([0,1],\Leb)$.  Let $f: \R \to \R$ be a bounded continuous function with $f|_{[-2,2]} = \id$, and let $X_i = [f(X_i^{(n)})]_{n \in \N} \in N$.  Then
		\begin{enumerate}[(1)]
			\item $\pi_{E,\cV}$ is injective on $Q$, and $\pi_{E,\cV}(Q)$ is an elementary submodel of $M$.
			\item $X_i$ does not depend on the particular choice of $f$.
			\item The elements $\pi_{E,\cV}(X_i)$ are freely independent of each other and of $\pi_{E,\cV}(Q)$.
			\item For each $k+\ell$-variable formula $\varphi$, indices $i_1$, \dots, $i_k \in I$, and elements $A_1$, \dots, $A_\ell \in (Q)_1$, the value of
			\[
			\varphi^{M}(\pi_{E,\cV}(X_{i_1}/2),\dots,\pi_{E,\cV}(X_{i_k}/2),\pi_{E,\cV}(A_1),\dots,\pi_{E,\cV}(A_\ell))
			\]
			does not depend on $\cV$ (it only depends on $\cU$).
		\end{enumerate}
	\end{prop}
	
	\begin{proof}
		(1) Note that $N = \prod_{n \to \cU} (L^\infty[0,1] \otimes \M_n)$ is a quotient of
		\[
		\tilde{N} = \bigoplus_{n \in \N} L^\infty[0,1] \otimes \M_n \cong L^\infty[0,1] \overline{\otimes} \bigoplus_{n \in \N} \M_n;
		\]
		let $q$ denote the quotient map, and let $\tilde{E}: \tilde{N} \to Z(\tilde{N)} \cong L^\infty[0,1] \overline{\otimes} \ell^\infty(\N)$ be the center-valued trace.  As in the proof of Proposition \ref{prop: distribution ultrapower}, we have
		\[
		M = N^{/E,\cV} \cong \tilde{N}^{/\tilde{E},\cV \circ q}.
		\]
		Note also that $\tilde{N}$ can be viewed as a direct integral of $M_n(\C)$ over the probability space $\bigsqcup_{n \in \N} [0,1]$ corresponding to the algebra $L^\infty[0,1] \overline{\otimes} \ell^\infty(\N)$ (after fixing a probability measure on $\N$ with full support, which, however, will play no role in the argument).  For $j = 1$, \dots, $k$, let $(A_j^{(n)})_{n \in \N} \in \bigoplus_{n \in \N} \M_n$ be a sequence of matrices bounded in operator norm, and let
		\[
		\tilde{A}_j = 1_{L^\infty[0,1]} \otimes \bigoplus_{n \in \N} A_j^{(n)} \in \tilde{N}.
		\]
		Let $\varphi$ be a formula with $k$ free variables.  By considering these with respect to the direct integral decomposition over $L^\infty[0,1] \overline{\otimes} \ell^\infty(\N)$ and applying Lemma \ref{direct-integral-case}, we see that
		\[
		\varphi_{\tilde{E}}^{\tilde{N}}(\tilde{A}_1,\dots,\tilde{A}_k) = 1_{L^\infty[0,1]} \otimes \bigoplus_{n \in \N} \varphi^{\M_n}(A_1^{(n)},\dots,A_k^{(n)}).
		\]
		Then by Theorem \ref{thm: continuous Los}, we have
		\begin{align*}
			\varphi^{M}(\pi_{\tilde{E},\cV \circ q}(\tilde{A}_1),\dots, \pi_{\tilde{E},\cV \circ q}(\tilde{A}_k)) &= \cV \circ q[\varphi_{\tilde{E}}^{\tilde{N}}(\tilde{A}_1,\dots,\tilde{A}_k)] \\
			&= \cV \circ q \left( 1_{L^\infty[0,1]} \otimes \bigoplus_{n \in \N} \varphi^{\M_n}(A_1^{(n)},\dots,A_k^{(n)}) \right) \\
			&= \cV \circ q \left( \bigoplus_{n \in \N} 1_{L^\infty[0,1]} \otimes \varphi^{\M_n}(A_1^{(n)},\dots,A_k^{(n)}) \right) \\
			&= \cV \left( 1_{N} \lim_{n \to \cU} \varphi^{\M_n}(A_1^{(n)},\dots,A_k^{(n)}) \right) \\
			&= \lim_{n \to \cU} \varphi^{\M_n}(A_1^{(n)},\dots,A_k^{(n)}) \\
			&= \varphi^{Q}(A_1,\dots,A_k),
		\end{align*}
		where $A_j$ is the image of $[A_j^{(n)}]_{n \in \N}$ in $Q$.  Note that $\pi_{\tilde{E},\cV \circ q}(\tilde{A}_j) = \pi_{E,\cV}(A_j)$, we have shown that
		\[
		\varphi^M(\pi_{E,\cV}(A_1),\dots,\pi_{E,\cV}(A_k)) = \varphi^Q(A_1,\dots,A_k),
		\]
		for all formulas $\varphi$ and $A_j \in Q$.  Taking $\varphi$ to be the distance, we see that $\pi_{E,\mathcal{V}}|_Q$ is isometric, hence injective.  And since this relation in fact holds for all formulas $\varphi$, the mapping $Q \to M$ is an elementary embedding.
		
		(2) It is well known that for GUE matrices, we have $\lim_{n \to \infty} \norm{X_i^{(n)}} = 2$ almost surely.  Suppose $f$ and $g$ are two bounded continuous functions that both equal identity on $[-2,2]$.  Given $\epsilon > 0$, there exists $\delta > 0$ such that $|f - g| < \epsilon$ on $[-2-\delta,2+\delta]$.  Thus, when $\norm{X_i^{(n)}} \leq 2 + \delta$, which happens for sufficiently large $n$ almost surely, we have $\norm{f(X_i^{(n)}) - g(X_i^{(n)})} \leq \epsilon$.  Therefore, $\limsup_{n \to \infty} \norm{f(X_i^{(n)}) - g(X_i^{(n)})}_{L^2([0,1],\M_n)} \leq \epsilon$, and since $\epsilon$ was arbitary, the limit is zero.  This implies that $(f(X_i^{(n)}))_{n \in \N}$ and $(g(X_i^{(n)}))_{n \in \N}$ produce the same element of $N = \prod_{n \to \cU} (L^\infty[0,1] \otimes \M_n)$.
		
		(3) Recall that to show free independence of $(\pi_{E,\cV}(X_i))_{i \in I}$ and $(\pi_{E,\cV}(Q))$, we need to show that for non-commutative polynomials $p_1$, \dots, $p_k$ and for $A_1$, \dots, $A_k \in Q$, we have
		\[
		\tau_{M} \left[ \prod_{j=1}^k (p_j((\pi_{E,\cV}(X_i))_{i\in I}) - \tau_{M}[p_j((\pi_{E,\cV}(X_i))_{i\in I})]) (\pi_{E,\cV}(A_j) - \tau_{M}[\pi_{E,\cV}(A_j)]) \right] = 0,
		\]
		where the product is understood to run from left to right.  We also have to prove similar claims for alternating strings that start with $A_1 - \tau(A_1)$ or end with $p_k((X_i)_{i \in I}) - \tau(p_k((X_i)_{i \in I}))$, but the proof in these cases, of course, is symmetrical, so we focus on the first case written above.  Since $\pi_{E,\cV}: N \to M$ is a $*$-homomorphism and $\tau_{M} = \cV \circ E$, it suffices to show that
		\[
		E \left[ \prod_{j=1}^k (p_j((X_i)_{i\in I}) - E[p_j((X_i)_{i\in I})]) (A_j - E(A_j)) \right] = 0 \text{ in } N.
		\]
		The latter is the element of the ultraproduct $\prod_{n \to \cU} (L^\infty[0,1] \otimes \M_n)$ given by the sequence
		\[
		\tr_n \left[ \prod_{j=1}^k (p_j((f(X_i^{(n)}))_{i\in I}) - \tr_n[p_j((f(X_i^{(n)}))_{i\in I})]) (A_j^{(n)} - \tr_n(A_j^{(n)})) \right]
		\]
		where $(A_j^{(n)})_{n \in \N}$ is a sequence representing $A_j$ in $Q$.  By Voiculescu's asymptotic freeness theorem (see e.g.\ \cite[Theorem 5.4.5]{AGZ2009}), this sequence of random variables converges to zero almost surely.  Since they are bounded, they also converge to zero in $L^2$ as $n \to \infty$.  Therefore, this sequence of bounded random variables agrees with the zero element in $N$, as desired.
		
		(4) This follows by a concentration of measure argument similar to \cite[Lemma 4.2]{FJP2023}, so we will proceed briefly here.  By (2), assume without loss of generality that $|f| \leq 2$ everywhere.  For $A_1$, \dots, $A_\ell \in (Q)_1$ and for each $j$, lift $A_j$ to a sequence $(A_j^{(n)})_{n \in \N}$ where $A_j^{(n)} \in (\M_n)_1$.  It suffices to check the claim on a dense set of formulas, such as the rational polynomial formulas in Lemma \ref{lem: separability}.  These formulas are easily seen to be Lipschitz on the unit ball.  Hence,
		\[
		\varphi^{\M_n}(f(X_{i_1}^{(n)})/2,\dots,f(X_{i_k}^{(n)})/2,A_1^{(n)},\dots,A_\ell^{(n)})
		\]
		is a Lipschitz function of $(X_{i_1}^{(n)}, \dots, X_{i_k}^{(n)})$.  Let $a_n$ be its expectation.  Then by Gaussian concentration of measure (see e.g.~\cite[Lemma 2.3.3]{AGZ2009}), we have that
		\[
		\lim_{n \to \infty} \norm{\varphi^{\M_n}(f(X_{i_1}^{(n)})/2,\dots,f(X_{i_k}^{(n)})/2,A_1^{(n)},\dots,A_\ell^{(n)}) - a_n 1}_{L^2[0,1]} = 0.
		\]
		Therefore, the sequence $\varphi^{\M_n}(f(X_{i_1}^{(n)})/2,\dots,f(X_{i_k}^{(n)})/2,A_1^{(n)},\dots,A_\ell^{(n)})$ agrees in $L^\infty[0,1]^{\cU}$ with the constant $\lim_{n \to \cU} a_n$.  It follows by similar reasoning to (1) that
		\begin{align*}
			\varphi^{M}(&\pi_{E,\cV}(X_{i_1}/2),\dots,\pi_{E,\cV}(X_{i_k}/2),\pi_{E,\cV}(A_1),\dots,\pi_{E,\cV}(A_\ell)) \\
			&= \cV \left[ [\varphi^{\M_n}(f(X_{i_1}^{(n)})/2,\dots,f(X_{i_k}^{(n)})/2,A_1^{(n)},\dots,A_\ell^{(n)})]_{n \in \N} \right] \\
			&= \cV (1 \cdot \lim_{n \to \cU} a_n) \\
			&= \lim_{n \to \cU} a_n.
		\end{align*}
		This is clearly independent of $\cV$, and only depends on $\cU$ and $A_1$, \dots, $A_\ell$.
	\end{proof}

	\bibliographystyle{amsalpha}
	\bibliography{ultrafiber}
	
\end{document}